\documentclass[12pt]{article}

\usepackage{amsmath}
\usepackage{amsfonts}
\usepackage{amssymb}
\usepackage{wasysym}
\usepackage{amsthm}
\usepackage[margin=1in]{geometry}

\title{On integral forms for vertex algebras associated with affine Lie algebras
and lattices}
\author{Robert McRae}
\date{}

    \theoremstyle{definition}\newtheorem{rema}{Remark}[section]
    \theoremstyle{plain}\newtheorem{propo}[rema]{Proposition}
    \newtheorem{theo}[rema]{Theorem}
    \newtheorem{defi}[rema]{Definition}
    \newtheorem{lemma}[rema]{Lemma}
    \newtheorem{corol}[rema]{Corollary}
    \theoremstyle{definition}\newtheorem{exam}[rema]{Example}

\begin{document}
\bibliographystyle{alpha}
\maketitle

\newcommand{\nordcirc}{\mbox{\tiny ${\circ\atop\circ}$}}
\numberwithin{equation}{section}

\begin{abstract}
\noindent We revisit the construction of integral forms for vertex (operator)
algebras $V_L$ based on even lattices $L$ using generators instead of bases, and
we construct integral forms for $V_L$-modules. We construct integral forms for
vertex (operator) algebras based on highest-weight modules for affine Lie
algebras and we exhibit natural generating sets. For vertex operator algebras in
general, we give conditions showing when an integral form contains the standard
conformal vector generating the Virasoro algebra. Finally, we study integral
forms in contragredients of modules for vertex algebras.
\end{abstract}

\section{Introduction}

While vertex algebras are ordinarily assumed to be vector spaces over
$\mathbb{C}$, or 
over any field of characteristic zero, the axioms, in particular the Jacobi 
identity, make sense for any commutative ring, and so it is natural to consider 
vertex algebras over $\mathbb{Z}$. In particular, it is natural to look for 
$\mathbb{Z}$-forms of vertex algebras over $\mathbb{C}$, by analogy with the 
construction of Lie algebras over $\mathbb{Z}$ using Chevalley bases. Given 
an integral form of a vertex algebra, one can then construct algebras over 
fields of prime characteristic $p$ by reducing structure constants mod $p$ and 
then extending 
scalars. 

Integral forms for vertex algebras, and especially for vertex 
algebras constructed from even lattices, have been studied previously in 
\cite{B}, \cite{P}, and \cite{DG}, and have been used in the modular moonshine program of Borcherds and Ryba (\cite{R}, \cite{BR1},  \cite{BR2}, \cite{GL}). In this paper, we continue the study of 
integral forms for vertex algebras, revisiting some known results with new 
methods and proving new results as well. Our approach is based on finding
generators for integral forms of vertex algebras. One advantage of this new
method is that it allows the construction of an integral form containing desired
elements without knowledge of the full structure in advance. This is 
particularly useful in constructing integral forms in vertex algebras based on 
affine Lie algebras and in constructing integral forms containing the standard 
conformal vector $\omega$ generating the Virasoro algebra. Further,
defining an integral form $V_\mathbb{Z}$ for a vertex algebra $V$ to be the
vertex subalgebra over $\mathbb{Z}$ generated by certain elements essentially
reduces the problem of proving that $V_\mathbb{Z}$ is in fact an integral form
of $V$ to the problem of showing that it is an integral form of $V$ as a vector
space. In the case of vertex algebras based on lattices, this is easier than
proving that an integral form of $V$ as a vector space is also a vertex
subalgebra, which is the method of proof used in \cite{P}.

In Section 2, we specify the classes of vertex algebras and modules that we will
be studying and state some basic facts about vertex algebra integral forms. In 
Section 3, we construct integral forms for vertex algebras based
on affine Lie algebras. In particular, we study affine Lie algebras based on a
finite dimensional simple Lie algebra $\mathfrak{g}$ using the integral form of
the universal enveloping algebra of an affine Lie algebra constructed in 
\cite{G} 
(see also \cite{M} and \cite{P}). For these integral forms, we exhibit a natural
set
of generating elements. The recent paper \cite{GL} has constructed integral forms in level $1$ affine Lie algebra vertex algebras using integral forms in lattice vertex algebras, but our results here can be applied to affine Lie algebra vertex algebras and modules of arbitrary integral level.

In Section 4, we study integral forms for vertex algebras $V_L$ based on even
lattices $L$. In his research announcement \cite{B}, Borcherds defined an
integral form for such a vertex algebra and exhibited a $\mathbb{Z}$-basis 
for this form. It has been proved in \cite{P} and \cite{DG} that this structure
is in fact an integral form, essentially by showing that the vertex algebra 
product of any two members
of Borcherds' $\mathbb{Z}$-basis is a $\mathbb{Z}$-linear combination of basis
elements. In Section 4, we provide an alternate proof by defining the integral
form to be the
vertex subalgebra over $\mathbb{Z}$ generated by a natural generating set and
then proving that the resulting structure is an integral form of the vector
space $V_L$. We also show that our definition of the integral form is equivalent
to the definition in \cite{B}. Further, we construct integral forms in modules
for lattice vertex
algebras. This problem is slightly more subtle than the problem of constructing
an integral form in the algebra due to the nature of a central extension of the
dual lattice $L^\circ$ that is needed to construct modules for $V_L$.

 In Section 5, we consider when an integral form for a vertex 
operator algebra contains the standard conformal element $\omega$ generating the
Virasoro algebra. For a vertex algebra
based on a lattice $L$ we solve this problem by proving that the integral form
constructed in Section 4 contains $\omega$ if and only if $L$ is self-dual (the ``if'' direction was observed in \cite{BR1}). More
generally, we give conditions under which an integral form of a
vertex operator algebra can be extended to include a multiple of $\omega$.
Finally, in Section 6, we consider the construction of integral forms in
contragredient modules, applying this to the situation in which the vertex
algebra has a non-degenerate invariant bilinear form. Some of these results have
already appeared in \cite{DG} in the context of invariant forms on vertex
algebras, but we formulate them more generally here in the context of
contragredient modules.

\paragraph{Acknowledgements}
This paper is part of my thesis work at Rutgers University; I am very 
grateful to my advisor James Lepowsky for many helpful discussions and 
encouragement. This research was partially supported by NSF grant 
DMS-0701176.

\section{Vertex algebras and integral forms}

In this paper, by vertex algebra we will mean the structure as defined by 
Borcherds in \cite{B} (see also \cite{LL} for an equivalent definition and 
proof of the equivalence). We use the definition of \cite{FLM} for vertex 
operator algebra (see also \cite{LL}); that is, a vertex operator algebra is 
a vertex algebra with a conformal vector $\omega$ satisfying the usual 
properties. We can weaken the notion of vertex operator algebra by dropping 
the grading restriction conditions on the $L(0)$-weight spaces: that is, we 
allow infinitely many (non-zero) spaces of negative weight and we allow 
weight spaces to be infinite dimensional. Such a structure is called a 
conformal vertex algebra in \cite{HLZ}. In this paper, we will use the 
definition of module for a vertex (operator) algebra or conformal vertex 
algebra as in \cite{LL} and \cite{HLZ}. In particular, if a module has a 
weight grading, the weights are allowed to be in $\mathbb{C}$.

The specific conformal vertex algebras and their modules that we will study in
this paper have an additional grading by an abelian group, in addition to the
weight grading. Thus we need the notion of strongly $A$-graded conformal vertex
algebra, $A$ an abelian group, as formulated in \cite{HLZ}: a conformal vertex
algebra $V$ is strongly $A$-graded if there is an additional grading by $A$,
\begin{equation}
 V=\coprod_{\alpha\in A} V^\alpha ,
\end{equation}
compatible with the weight grading in the sense that for any $\alpha\in A$,
\begin{equation}
 V^\alpha=\coprod_{n\in\mathbb{Z}} V^\alpha_n,
\end{equation}
where $V^\alpha_n =V^\alpha\cap V_n$ for  $n\in\mathbb{Z}$. Further, we require
that $\mathbf{1}\in V^0_0$, $\omega\in V^0_2$, and if $v\in V^\alpha$, then
\begin{equation}\label{agradingops}
 Y(v,x)V^\beta\subseteq V^{\alpha+\beta}[[x,x^{-1}]].
\end{equation}
Finally, we want the grading restriction conditions: for any fixed $\alpha\in
A$, $V^\alpha_n=0$ for $n$ sufficiently negative and dim $V^\alpha_n<\infty$ for
any $n$. A module for a strongly $A$-graded conformal vertex algebra $V$ is a
module $W$ for $V$ as conformal vertex algebra with an additional $B$-grading,
$B$ an abelian group containing $A$. The $B$-grading of $W$ and the
$\mathbb{C}$-grading by conformal weight are compatible, the grading restriction
conditions hold, and $V^\alpha$ maps $W^\beta$ into $W^{\alpha+\beta}$, as
above.

The notion of vertex algebra over $\mathbb{Z}$ makes sense because all
numerical coefficients in the formal delta functions appearing in the
Jacobi identity \cite{FLM} are integers. For convenience, we will call vertex
algebras over $\mathbb{Z}$ \textit{vertex rings} in this paper, and we will call
$\mathbb{Z}$-subalgebras of vertex algebras \textit{vertex subrings}.  Recall
that an
integral form
of a vector space $V$ is a free abelian group $V_\mathbb{Z}$ such that
the canonical map 
\begin{equation}
 \mathbb{C}\otimes_{\mathbb{Z}}V_\mathbb{Z}\rightarrow V,
\end{equation}
 given by 
\begin{equation}  
 c\otimes_{\mathbb{Z}}v\mapsto cv
\end{equation}
 for $c\in\mathbb{C}$ and $v\in V$, is an isomorphism. That is, $V_\mathbb{Z}$
is the $\mathbb{Z}$-span of a basis of $V$.  Thus given a vertex
algebra $V$ over $\mathbb{C}$ (or any field of characteristic zero),
we can define an integral form of $V$ to be a vertex subring
$V_\mathbb{Z}\subseteq V$ which is  an integral form
 of the vector space $V$.   In other words,
$V_\mathbb{Z}$ is the $\mathbb{Z}$-span of a basis for $V$, it
contains the vacuum vector $\mathbf{1}$, and it is closed under vertex
algebra products. The notion of module for a vertex algebra over
$\mathbb{Z}$ also makes sense. Thus if $V$ is a vertex algebra over
$\mathbb{C}$ having an integral form $V_\mathbb{Z}$ and $W$ is a
$V$-module, we can define an integral form for $W$ to be a
$V_\mathbb{Z}$-submodule $W_\mathbb{Z}\subseteq W$ which is an integral
form of the vector space $W$. That is, $W_\mathbb{Z}$ is the
$\mathbb{Z}$-span of a basis for $W$, and it is preserved by vertex
operators from $V_\mathbb{Z}$.

If $V$ is also a vertex operator algebra or conformal vertex algebra,
and so has a conformal element $\omega$, we do not require an integral
form of $V$ to contain $\omega$. Such a requirement would disallow
many interesting integral forms. However, in this paper, we will
require an integral form $V_\mathbb{Z}$ of $V$ to be compatible with
the weight grading:
\begin{equation}\label{compatibility}
 V_\mathbb{Z}=\coprod_{n\in\mathbb{Z}} V_n\cap V_\mathbb{Z},
\end{equation}
where $V_n$ is the weight space with $L(0)$-eigenvalue $n$. Moreover, if $V$ is
a strongly $A$-graded conformal vertex algebra, we require $V_\mathbb{Z}$ to be
compatible with the $A\times\mathbb{C}$-gradation:
\begin{equation}\label{compatibility2}
 V_{\mathbb{Z}}=\coprod_{\alpha\in A,\, n\in\mathbb{Z}} V^\alpha_n\cap
V_\mathbb{Z}.
\end{equation}
We require the analogues of these compatibility conditions for modules for a
conformal vertex algebra or strongly $A$-graded conformal vertex algebra. One
could
assume different requirements on the relation of $\omega$ to
$V_\mathbb{Z}$; for instance, in \cite{DG} it is required that some
integer multiple of $\omega$ be in $V_\mathbb{Z}$.

We conclude this section with the following useful general results on vertex 
rings:
\begin{propo}\label{zgen}
 Suppose $V$ is a vertex algebra; for a subset $S$ of $V$, denote by 
$\left\langle S\right\rangle_\mathbb{Z}$ the vertex subring
generated by $S$. Then  $\left\langle S\right\rangle_\mathbb{Z}$ is the 
$\mathbb{Z}$-span of coefficients of products of the form
\begin{equation}\label{zspan}
 Y(u_1,x_1)\ldots Y(u_k,x_k)\mathbf{1}
\end{equation}
where $u_1,\ldots u_k\in S$. Moreover, if $W$ is a $V$-module, the 
$\left\langle S\right\rangle_\mathbb{Z}$-submodule generated by a subset 
$T$ of $W$ is the $\mathbb{Z}$-span of coefficients of products of the 
form
\begin{equation}
 Y(u_1,x_1)\ldots Y(u_k,x_k)w
\end{equation}
where $u_1,\ldots u_k\in S$ and $w\in T$.
\end{propo}
\begin{proof}
 Use the proof of Proposition 3.9.3 in \cite{LL}, noting that all numerical 
coefficients in the Jacobi identity are integers. 
\end{proof}

\begin{rema}
 Proposition \ref{zgen} applies even if $\left\langle 
S\right\rangle_\mathbb{Z}$ 
and the 
$\left\langle S\right\rangle_\mathbb{Z}$-submodule generated by $T$ are not 
integral forms of their respective vector spaces.
\end{rema}

\begin{propo}\label{spanofvacuum}
 Suppose $V$ is a vertex algebra with integral form $V_\mathbb{Z}$. Then 
$V_\mathbb{Z}\cap\mathbb{C}\mathbf{1}=\mathbb{Z}\mathbf{1}$.
\end{propo}
\begin{proof}
 Since $\mathbf{1}\in V_\mathbb{Z}$, it is clear that 
$\mathbb{Z}\mathbf{1}\subseteq V_\mathbb{Z}\cap\mathbb{C}\mathbf{1}$. On the 
other hand, since $V_\mathbb{Z}$ is the $\mathbb{Z}$-span of a basis $\lbrace 
v_i\rbrace $ for $V$, $\mathbf{1}=\sum_i n_i v_i$ where 
$n_i\in\mathbb{Z}$. If $c\mathbf{1}=\sum_i c 
n_i v_i\in V_\mathbb{Z}\cap\mathbb{C}\mathbf{1}$ for $c\in\mathbb{C}$, then $c
n_i\in\mathbb{Z}$ 
for each $i$, so $c\in\mathbb{Q}$. If $c\notin\mathbb{Z}$, by subtracting 
off an integer multiple of $\mathbf{1}$ from $c\mathbf{1}$, we may assume 
$0<c<1$. Since 
$V_\mathbb{Z}$ is closed under vertex operators and $Y(\mathbf{1},x)=1_V$, we 
see that $c^n\mathbf{1}\in V_\mathbb{Z}$ for all $n\geq 0$, contradicting the 
assumption that $V_\mathbb{Z}$ is an integral form of $V$. Thus 
$c\in\mathbb{Z}.$
\end{proof}

\section{Integral forms in vertex algebras based on affine Lie algebras}

We now recall the construction of vertex algebras based on affine Lie algebras
(see \cite{FZ} and \cite{LL} for more details). Suppose that $\mathfrak{g}$ is a
finite dimensional Lie algebra, with symmetric invariant form $\langle\cdot
,\cdot\rangle$. Then we can form the affine Lie algebra
\begin{equation}
 \widehat{\mathfrak{g}}=\mathfrak{g}\otimes\mathbb{C}[t,t^{-1}]\oplus\mathbb{C}
\mathbf{k}
\end{equation}
where $\mathbf{k}$ is central and all other brackets are determined by
\begin{equation}\label{affcomm}
 [a\otimes t^m,b\otimes t^n]=[a,b]\otimes t^{m+n}+m\langle
a,b\rangle\delta_{m+n,0}\mathbf{k},
\end{equation}
where $a,b\in\mathfrak{g}$ and $m,n\in\mathbb{Z}$. The affine Lie algebra has 
the decomposition
\begin{equation}
 \widehat{\mathfrak{g}}=\widehat{\mathfrak{g}}_+\oplus\widehat{\mathfrak{g}}
_0\oplus\widehat{\mathfrak{g}}_-
\end{equation}
where 
\begin{equation}
 \widehat{\mathfrak{g}}_\pm=\coprod_{n\in\pm\mathbb{Z}_+}\mathfrak{g}\otimes
t^n,\;\;\;\widehat{\mathfrak{g}}_0=\mathfrak{g}\otimes
t^0\oplus\mathbb{C}\mathbf{k}.
\end{equation}
If $U$ is a finite dimensional $\mathfrak{g}$-module, then $U$ becomes a
$\widehat{\mathfrak{g}}_+\oplus\widehat{\mathfrak{g}}_0$-module on which
$\widehat{\mathfrak{g}}_+$ acts trivially, $\mathfrak{g}\otimes t^0$ acts as
$\mathfrak{g}$, and $\mathbf{k}$ acts as some scalar $\ell\in\mathbb{C}$. Then
we can form the generalized Verma module
\begin{equation}
 V_{\widehat{\mathfrak{g}}}(\ell,U)=U(\widehat{\mathfrak{g}})\otimes_{U(\widehat
{\mathfrak{g}}_+\oplus\widehat{\mathfrak{g}}_0)}U,
\end{equation}
which is a $\widehat{\mathfrak{g}}$-module of level $\ell$; it has a unique
maximal submodule and thus a unique irreducible quotient,
$L_{\widehat{\mathfrak{g}}}(\ell,U)$. We use
$V_{\widehat{\mathfrak{g}}}(\ell,0)$ to denote
$V_{\widehat{\mathfrak{g}}}(\ell,U)$ where $U$ is the trivial 
$\mathfrak{g}$-module, and
$L_{\widehat{\mathfrak{g}}}(\ell,0)$ to denote its irreducible quotient. If we
write the trivial $\mathfrak{g}$-module with $\mathbf{k}$ acting as $\ell$ as 
$\mathbb{C}
1_\ell$, and then write $\mathbf{1}=1\otimes 1_\ell\in
V_{\widehat{\mathfrak{g}}}(\ell,0)$, then we can write
\begin{equation}
 V_{\widehat{\mathfrak{g}}}(\ell,0)=U(\widehat{\mathfrak{g}}_-)\mathbf{1},
\end{equation}
so $V_{\widehat{\mathfrak{g}}}(\ell,0)$ is spanned by elements of the form
\begin{equation}\label{liealgspan}
 a_1(-n_1)\ldots a_k(-n_k)\mathbf{1}
\end{equation}
where $a_i\in\mathfrak{g}$, $n_i\in\mathbb{Z}_+$, and we use the notation $a(n)$
to denote the action of $a\otimes t^n$ on $\widehat{\mathfrak{g}}$-modules.

The generalized Verma module $V_{\widehat{\mathfrak{g}}}(\ell,0)$ is a vertex
algebra, as are all its quotients, including
$L_{\widehat{\mathfrak{g}}}(\ell,0)$, which is a simple vertex algebra. The
vertex algebra structure is determined by
\begin{equation}\label{liealgops}
 Y(a(-1)\mathbf{1}, x)=\sum_{n\in\mathbb{Z}} a(n) x^{-n-1}.
\end{equation}
(The elements $a(-1)\mathbf{1}$ generate $V_{\widehat{\mathfrak{g}}}(\ell,0)$,
and thus also its quotients, as vertex algebras.) As long as $\ell\neq -h$,
where $h$ is the dual Coxeter number of $\mathfrak{g}$,
$V_{\widehat{\mathfrak{g}}}(\ell,0)$ is also a vertex operator algebra, with
conformal vector given by
\begin{equation}
 \omega=\frac{1}{2(\ell+h)}\sum_{i=1}^{\mathrm{dim}\,\mathfrak{g}}
u_i(-1)u_i'(-1)\mathbf{1},
\end{equation}
where $\left\lbrace u_i\right\rbrace $ is any basis of $\mathfrak{g}$ and
$\left\lbrace u_i'\right\rbrace $ is the corresponding dual basis with respect
to the form $\langle\cdot ,\cdot\rangle$. The generalized Verma modules 
$V_{\widehat{\mathfrak{g}}}(\ell,U)$ for finite-dimensional 
$\mathfrak{g}$-modules $U$ are $V_{\widehat{\mathfrak{g}}}(\ell,0)$-modules, as 
well as all quotients of $V_{\widehat{\mathfrak{g}}}(\ell, U)$;  the irreducible 
$V_{\widehat{\mathfrak{g}}}(\ell,0)$-modules  consist precisely of the
$\widehat{\mathfrak{g}}$-modules $L_{\widehat{\mathfrak{g}}}(\ell,U)$ where $U$
is a finite-dimensional irreducible $\mathfrak{g}$-module.

In order to construct integral forms in these vertex algebras and modules, we
will assume that $\mathfrak{g}$ has an integral form. That is, we assume that
$\mathfrak{g}$ has a basis whose $\mathbb{Z}$-span $\mathfrak{g}_\mathbb{Z}$ is
closed under the bracket and such that the form $\langle\cdot,\cdot\rangle$ is
integer-valued on $\mathfrak{g}_\mathbb{Z}$. Then $\widehat{\mathfrak{g}}$ also
has the integral form
\begin{equation}
 \widehat{\mathfrak{g}}_\mathbb{Z}=\mathfrak{g}_\mathbb{Z}\otimes_{\mathbb{Z}}
\mathbb{Z}[t,t^{-1}]\oplus\mathbb{Z}\mathbf{k}.
\end{equation}
For example, if $\mathfrak{h}$ is a finite-dimensional abelian Lie algebra with
a symmetric non-degenerate form, any full-rank integral lattice
$L\subseteq\mathfrak{h}$ is an integral form of $\mathfrak{h}$. If 
$\mathfrak{g}$
is a finite-dimensional simple Lie algebra, we can take for
$\mathfrak{g}_\mathbb{Z}$ the $\mathbb{Z}$-span of a Chevalley basis of
$\mathfrak{g}$. (There do exist symmetric invariant non-degenerate
integer-valued bilinear forms on such a $\mathfrak{g}_\mathbb{Z}$ because the
Killing form $\kappa(a,b)=\mathrm{Tr}\,(\mathrm{ad}\,a\,\mathrm{ad}\,b)$ is
integral on $\mathfrak{g}_\mathbb{Z}$.)

Now suppose that $U_{\mathbb{Z}}(\widehat{\mathfrak{g}})$ is an integral form 
of 
the universal enveloping algebra $U(\widehat{\mathfrak{g}})$ such that
\begin{equation}\label{PBWdecomp}
 U_{\mathbb{Z}}(\widehat{\mathfrak{g}})= 
U_{\mathbb{Z}}(\widehat{\mathfrak{g}}_-) 
U_{\mathbb{Z}}(\widehat{\mathfrak{g}}_0) 
U_{\mathbb{Z}}(\widehat{\mathfrak{g}}_+),
\end{equation}
where $ U_{\mathbb{Z}}(\widehat{\mathfrak{g}}_\pm)$ and $ 
U_{\mathbb{Z}}(\widehat{\mathfrak{g}}_0)$  are integral forms of  
$U(\widehat{\mathfrak{g}}_\pm)$ and $U(\widehat{\mathfrak{g}}_0)$, respectively. 
We also assume that $U_{\mathbb{Z}}(\widehat{\mathfrak{g}}_\pm)$ are graded in 
the sense that
\begin{equation}\label{universalgrading}
 U_{\mathbb{Z}}(\widehat{\mathfrak{g}}_\pm)=\coprod_{n\in\mathbb{Z}} 
U_{\mathbb{Z}}(\widehat{\mathfrak{g}}_\pm)\cap U(\widehat{\mathfrak{g}}_\pm)_n
\end{equation}
where  $U(\widehat{\mathfrak{g}}_\pm)_n$ is the space spanned by monomials of 
the form
\begin{equation}
 (g_1\otimes t^{n_1})\cdots (g_k\otimes t^{n_k})
\end{equation}
where  $g_1,\ldots,g_k\in\mathfrak{g}$ and $n=n_1+\ldots +n_k$.
We want to use such a 
$U_{\mathbb{Z}}(\widehat{\mathfrak{g}})$ to obtain integral forms in 
generalized 
Verma modules $V_{\widehat{\mathfrak{g}}}(\ell,U)$ and their irreducible 
quotients when $\ell\in\mathbb{Z}$.
\begin{rema}\label{PBWintform}
 Perhaps the simplest way to obtain such a $ 
U_{\mathbb{Z}}(\widehat{\mathfrak{g}})$ is to take a $\mathbb{Z}$-basis 
$\lbrace 
u_i\rbrace$ of $\mathfrak{g}_{\mathbb{Z}}$ and extend it to a 
$\mathbb{Z}$-basis 
$\lbrace u_i\otimes t^n,\mathbf{k}\rbrace$ of 
$\widehat{\mathfrak{g}}_\mathbb{Z}$. Then $ 
U_{\mathbb{Z}}(\widehat{\mathfrak{g}})$ may be defined as the $\mathbb{Z}$-span 
of ordered monomials in these basis elements, ordered in such a way that the 
decomposition (\ref{PBWdecomp}) holds. This is certainly an integral form of 
$U(\widehat{\mathfrak{g}})$ as a vector space, and in fact it is an integral 
form of $ U(\widehat{\mathfrak{g}})$ as an associative algebra since 
it includes $1$ and since a product of two ordered monomials is an integral 
combination of ordered monomials: this is because in any unordered monomial, 
the 
order of any two basis elements of $\widehat{\mathfrak{g}}_\mathbb{Z}$ can be 
switched at the cost of a commutator term which is an integral linear 
combination of monomials of lower degree.
\end{rema}

By an integral form $U_\mathbb{Z}$ of a $\widehat{\mathfrak{g}}_0$-module $U$, 
we will mean an integral form of $U$ as a vector space which is invariant under 
$U_\mathbb{Z}(\widehat{\mathfrak{g}}_0)$, where 
$U_\mathbb{Z}(\widehat{\mathfrak{g}}_0)$ is as in (\ref{PBWdecomp}). For 
example, if $U_\mathbb{Z}(\widehat{\mathfrak{g}})$ is constructed as in Remark 
\ref{PBWintform} and $U$ as a vector space has an integral form $U_\mathbb{Z}$ 
which is a $\mathfrak{g}_\mathbb{Z}$-module, and $\mathbf{k}$ acts on $U$ as an 
integer, then $U_\mathbb{Z}$ is an integral form of $U$.

\begin{theo}\label{intformvect}
 Suppose $U_\mathbb{Z}$ is an integral form of a finite-dimensional 
$\widehat{\mathfrak{g}}_0$-module $U$ on which $\mathbf{k}$ acts as $\ell$, and 
suppose $W$ is $V_{\widehat{\mathfrak{g}}}(\ell, U)$ or its irreducible quotient 
$L_{\widehat{\mathfrak{g}}}(\ell,U)$. Then $W_\mathbb{Z} 
=U_\mathbb{Z}(\widehat{\mathfrak{g}})U_\mathbb{Z}$ is an integral form of $W$ as 
a vector space, and $W_{\mathbb{Z}}$ is compatible with the conformal weight 
grading of $W$. Moreover, $W_\mathbb{Z}$ is invariant under 
$U_\mathbb{Z}(\widehat{\mathfrak{g}})$.  
\end{theo}
\begin{proof}
 By definition, $W_\mathbb{Z}$ is invariant under 
$U_\mathbb{Z}(\widehat{\mathfrak{g}})$. Also, since 
$\widehat{\mathfrak{g}}_+\cdot U =0$ and since 
$U_\mathbb{Z}(\widehat{\mathfrak{g}}_0) U_\mathbb{Z}\subseteq U_\mathbb{Z}$,
 \begin{equation}
U_\mathbb{Z}(\widehat{\mathfrak{g}})U_\mathbb{Z}=U_{\mathbb{Z}}(\widehat{
\mathfrak{g}}_-) 
U_{\mathbb{Z}}(\widehat{\mathfrak{g}}_0) 
U_{\mathbb{Z}}(\widehat{\mathfrak{g}}_+)U_\mathbb{Z}=U_{\mathbb{Z}}(\widehat{
\mathfrak{g}}_-)U_\mathbb{Z}.
 \end{equation}
Since we assume $U_\mathbb{Z}(\widehat{\mathfrak{g}}_-)$ is graded as in 
(\ref{universalgrading}),  $W_\mathbb{Z}$ is graded by conformal weight and the 
intersection of $W_\mathbb{Z}$ with each weight space is spanned by finitely 
many vectors. In fact, an upper bound on the number of vectors required to span 
the weight space of weight $n$ higher than the lowest weight space is 
$\mathrm{dim}\,U(\widehat{\mathfrak{g}}_-)_n\cdot\mathrm{dim}\,U<\infty$. 
Moreover, since $W=U(\widehat{\mathfrak{g}}_-)U$, it follows that $W_\mathbb{Z}$ 
spans $W$ as a vector space. To prove the theorem, that is, to prove that $W$ is 
linearly isomorphic to $\mathbb{C}\otimes_{\mathbb{Z}}W_\mathbb{Z}$,
we need to show that if a set of vectors in $W_\mathbb{Z}$ is linearly 
independent over $\mathbb{Z}$, then they are also linearly independent over 
$\mathbb{C}$.

In the case that $W=V_{\widehat{\mathfrak{g}}}(\ell,U)$, then $W_\mathbb{Z}$ is 
isomorphic as a $\mathbb{Z}$-module to 
$U_\mathbb{Z}(\widehat{\mathfrak{g}}_-)\otimes_{\mathbb{Z}} U_\mathbb{Z}$, which 
is a free graded $\mathbb{Z}$-module whose weight spaces have rank equal to the 
dimension of the weight spaces of $W\cong 
U(\widehat{\mathfrak{g}}_-)\otimes_\mathbb{C} U$ since 
$U_\mathbb{Z}(\widehat{\mathfrak{g}}_-)$ is a graded integral form of  
$U(\widehat{\mathfrak{g}}_-).$ This proves the theorem in this case.

In the case that $W=L_{\widehat{\mathfrak{g}}}(\ell,U)$, we use a slight 
generalization of the proof of Theorem 11.3 in \cite{G}, which is an analogue of 
Lemma 12 in \cite{S}; see also Theorem 27.1 in \cite{H}. We first observe that 
$W_\mathbb{Z}\cap U=U_\mathbb{Z}$, and we use $\lbrace u_i\rbrace $ to denote a 
$\mathbb{Z}$-base of $U_\mathbb{Z}$. Next, we observe that if $w\in 
W_\mathbb{Z}$ is non-zero, there is some $y\in 
U_{\mathbb{Z}}(\widehat{\mathfrak{g}}_+)$ such that the component of $y\cdot w$ 
in $U$ is non-zero; this is because otherwise $w$ would generate a proper 
$\widehat{\mathfrak{g}}$-submodule in $L_{\widehat{\mathfrak{g}}}(\ell, U)$, 
which is impossible. Note that the component of $y\cdot w$ in $U$ is in 
$U_\mathbb{Z}$ since $W_\mathbb{Z}$ is graded and invariant under 
$U_\mathbb{Z}(\widehat{\mathfrak{g}})$.

Now we prove that $W_\mathbb{Z}$ is an integral form of $W$ by contradiction. 
Suppose that $\lbrace w_i\rbrace_{i=1}^k$ is a minimal set contained in 
$W_\mathbb{Z}$ that is linearly independent over $\mathbb{Z}$ but not over 
$\mathbb{C}$; note that for this to happen, $k\geq 2$. Then
\begin{equation}
 \sum_{i=1}^k c_i w_i=0
\end{equation}
for $c_i\in\mathbb{C}^\times$. Take $y\in U(\widehat{\mathfrak{g}}_+)$ such that 
the component $(y\cdot w_1)_U$ of $y\cdot w_1$ in $U_\mathbb{Z}$ is non-zero. In 
particular, for some basis element $u_j$ of $U_\mathbb{Z}$, there is a non-zero 
integer $m_1$ such that the projection of $(y\cdot w_1)_U$ to $\mathbb{Z}u_j$ 
(with respect to the basis $\lbrace u_i\rbrace$) is $m_1 u_j$. Also, for $2\leq 
i\leq k$, the projection of $(y\cdot w_i)_U$ to $\mathbb{Z} u_j$ is $m_i u_j$ 
for some integer $m_i$. Thus, since
\begin{equation}
 0=y\cdot\sum_{i=1}^k c_i w_i=\sum_{i=1}^k c_i(y\cdot w_i),
\end{equation}
we obtain
\begin{equation}
 \sum_{i=1}^k c_i m_i=0
\end{equation}
 because $\lbrace u_i\rbrace$ forms a basis for $U$ over $\mathbb{C}$ as well as 
a basis for $U_\mathbb{Z}$ over $\mathbb{Z}$. Thus
\begin{equation}
 0=m_1\left(\sum_{i=1}^k c_i w_i\right)-\left(\sum_{i=1}^k c_i m_i\right) 
w_1=\sum_{i=1}^k (m_1 c_i w_i-m_i c_i w_1)=\sum_{i=2}^k c_i (m_1 w_i-m_i w_1)
\end{equation}
Since $m_1\neq 0$, the vectors $m_1 w_i-m_i w_1$ for $2\leq i\leq k$ are thus in 
$W_\mathbb{Z}$, linearly independent over $\mathbb{Z}$, and linearly dependent 
over $\mathbb{C}$, which contradicts the minimality of $\lbrace 
w_i\rbrace_{i=1}^k$.
\end{proof}

We next prove a general result on vertex algebraic integral forms that applies 
to any finite-dimensional Lie algebra $\mathfrak{g}$ having an integral form.

\begin{propo}\label{basicaffint}
 Suppose $\mathfrak{g}_\mathbb{Z}$ is an integral form of $\mathfrak{g}$, 
$U_\mathbb{Z}(\widehat{\mathfrak{g}})$ is constructed as in Remark 
\ref{PBWintform}, and $\ell\in\mathbb{Z}$. Then the integral form 
$V_{\widehat{\mathfrak{g}}}(\ell,0)_\mathbb{Z}$ given by Theorem 
\ref{intformvect} is the vertex subring generated by the vectors 
$a(-1)\mathbf{1}$ for $a\in\mathfrak{g}_\mathbb{Z}$. Moreover, if $U$ is a 
finite-dimensional $\mathfrak{g}$-module with integral form $U_\mathbb{Z}$, 
$V_{\widehat{\mathfrak{g}}}(\ell,U)_\mathbb{Z}$ and 
$L_{\widehat{\mathfrak{g}}}(\ell,U)_\mathbb{Z}$ are the 
$V_{\widehat{\mathfrak{g}}}(\ell,0)_\mathbb{Z}$-modules generated by 
$U_\mathbb{Z}$.
\end{propo}
\begin{proof}
 Let $W$ be the module $V_{\widehat{\mathfrak{g}}}(\ell,U)$ or 
$L_{\widehat{\mathfrak{g}}}(\ell,U)$, where $U$ is possibly $\mathbb{C}1_\ell$. 
From the construction in Remark \ref{PBWintform},
 \begin{equation}
  W_\mathbb{Z}=U_\mathbb{Z}(\widehat{\mathfrak{g}}_-) U_\mathbb{Z}
 \end{equation}
is the integral span of vectors of the form
\begin{equation}\label{propspan}
 a_1(-n_1)\cdots a_k(-n_k)u
\end{equation}
where $a_i\in\mathfrak{g}_\mathbb{Z}$, $n_i>0$, and $u\in U_\mathbb{Z}$. On the 
other hand, since 
\begin{equation}
 Y(a(-1)\mathbf{1},x)=\sum_{n\in\mathbb{Z}} a(n) x^{-n-1}
\end{equation}
for $a\in\mathfrak{g}$, Proposition \ref{zgen} implies that the vertex subring 
generated by the $a(-1)\mathbf{1}$ is the integral span of vectors of the form
\begin{equation}\label{propspan2}
 a_1(n_1)\cdots a_k(n_k)\mathbf{1}
\end{equation}
for $a_i\in\mathfrak{g}_\mathbb{Z}$ and $n_i\in\mathbb{Z}$. In the case that 
$U=\mathbb{C}1_\ell$, this is the same as (\ref{propspan}) because 
$a(n)\mathbf{1}=0$ if $n>0$ and any $a_i(n_i)$ occurring in (\ref{propspan2}) 
with $n_i>0$ can be moved to the right using the commutation relations 
(\ref{affcomm}). This proves the first assertion of the proposition, and the 
second follows similarly.
\end{proof}

\begin{corol}
 In the setting of Proposition \ref{basicaffint}, 
$L_{\widehat{\mathfrak{g}}}(\ell,0)_\mathbb{Z}$ is the vertex ring generated by 
vectors of the form $a(-1)\mathbf{1}$ for $a\in\mathfrak{g}_\mathbb{Z}$.
\end{corol}
\begin{proof}
 This follows because $L_{\widehat{\mathfrak{g}}}(\ell,0)_\mathbb{Z}$ is the integral span 
of vectors of the form (\ref{propspan2}), which is also the vertex subring 
generated by vectors of the form $a(-1)\mathbf{1}$ for 
$a\in\mathfrak{g}_\mathbb{Z}$, just as in the proof of Proposition 
\ref{basicaffint}.
\end{proof}

If $\mathfrak{g}$ is a finite-dimensional simple Lie algebra, there is another
way to obtain integral structure in modules for $\widehat{\mathfrak{g}}$ 
(\cite{G}; see also \cite{M}, \cite{P}). Consider a Chevalley basis for 
$\mathfrak{g}$ where for a root $\alpha$, $x_\alpha$ is the basis vector of the 
corresponding root space. Then the subring 
$U_\mathbb{Z}(\widehat{\mathfrak{g}})\subseteq U(\widehat{\mathfrak{g}})$ 
generated by the elements $(x_{\alpha}\otimes t^n)^k/k!$ for $k\geq 0$, 
$n\in\mathbb{Z}$, and $\alpha $ a root of $\mathfrak{g}$ is an integral form of 
$U(\widehat{\mathfrak{g}})$ which satisfies (\ref{PBWdecomp}) and 
(\ref{universalgrading}); moreover, if $U$ is a finite-dimensional irreducible 
$\mathfrak{g}$-module with highest weight vector $v_0$, then 
$U_\mathbb{Z}(\widehat{\mathfrak{g}}_0)\cdot v_0$ is an integral form of $U$ as 
long as $\mathbf{k}$ acts as an integer $\ell$; here 
$U_\mathbb{Z}(\widehat{\mathfrak{g}}_0)=U_\mathbb{Z}(\widehat{\mathfrak{g}})\cap 
U(\widehat{\mathfrak{g}}_0)$. Consequently any finite-dimensional 
$\mathfrak{g}$-module has a $U_\mathbb{Z}(\widehat{\mathfrak{g}}_0)$-invariant integral 
form since finite-dimensional $\mathfrak{g}$-modules are completely reducible. 
Thus we can apply Theorem \ref{intformvect} and conclude that for any 
finite-dimensional $\mathfrak{g}$-module $U$, the 
$\widehat{\mathfrak{g}}$-modules $V_{\widehat{\mathfrak{g}}}(\ell, U)$ and 
$L_{\widehat{\mathfrak{g}}}(\ell, U)$ when $\ell\in\mathbb{Z}$ have 
$U_\mathbb{Z}(\widehat{\mathfrak{g}})$-invariant integral forms which are 
compatible with the conformal weight gradings.

We can now show that the integral forms obtained using this 
$U_\mathbb{Z}(\widehat{\mathfrak{g}})$ for a finite-dimensional simple Lie 
algebra $\mathfrak{g}$ have vertex algebraic integral structure:

\begin{theo}\label{affalgzform}
 Suppose $\ell\in\mathbb{Z}$; the integral form 
$V_{\widehat{\mathfrak{g}}}(\ell,0)_\mathbb{Z}$ is
the vertex subring of $V_{\widehat{\mathfrak{g}}}(\ell,0)$ generated by the
vectors $\frac{x_\alpha(-1)^k}{k!}\mathbf{1}$ where $k\geq 0$ and $x_\alpha$ is
the root vector corresponding to the root $\alpha$ in the chosen Chevalley basis
of $\mathfrak{g}$. Moreover, if $U$ is a finite-dimensional 
$\mathfrak{g}$-module with integral form $U_\mathbb{Z}$ and $W$ is 
$V_{\widehat{\mathfrak{g}}}(\ell,U)$ or $L_{\widehat{\mathfrak{g}}}(\ell,U)$, 
then $W_\mathbb{Z}$ is the 
 $V_{\widehat{\mathfrak{g}}}(\ell,0)_\mathbb{Z}$-module generated by 
$U_\mathbb{Z}$.
\end{theo}
\begin{proof}
 Since  $U_\mathbb{Z}(\widehat{\mathfrak{g}})$ is generated as a ring by the
divided powers $(x_{\alpha}\otimes t^n)^k/k!$ where $\alpha $ is a root of
$\mathfrak{g}$ and $k\geq 0$, we can express
$W_\mathbb{Z}=U_\mathbb{Z}(\widehat{\mathfrak{g}})\cdot U_\mathbb{Z}$ as the 
$\mathbb{Z}$-span of products of the form
\begin{equation}\label{affintspan2}
 \frac{x_{\alpha_1}(m_1)^{k_1}}{k_1!}\cdots
\frac{x_{\alpha_n}(m_n)^{k_n}}{k_n!}\cdot u,
\end{equation}
where the $\alpha_i$ are roots of $\mathfrak{g}$, $m_i\in\mathbb{Z}$,
$k_i\geq 0$, and $u\in U_\mathbb{Z}$ (where $U$ could be $\mathbb{C}1_\ell$, in 
which case $u=\mathbf{1}$). By Proposition
\ref{zgen}, we need to show that the $\mathbb{Z}$-span of such products equals
the $\mathbb{Z}$-span of coefficients of products of the form
\begin{equation}\label{affintspan}
 Y\left(\frac{x_{\alpha_1}(-1)^{k_1}}{k_1!},x_1\right)\cdots
Y\left(\frac{x_{\alpha_n}(-1)^{k_n}}{k_n!},x_n\right) u.
\end{equation}

First we will analyze the vertex operator associated to a generator
$x_\alpha(-1)^k/k!,$ where $\alpha$ is any root of $\mathfrak{g}$. First,
observe that for any $m,n\in\mathbb{Z}$,
\begin{equation}
 [x_\alpha(m),x_\alpha(n)]=[x_\alpha,x_\alpha](m+n)+\langle
x_\alpha,x_\alpha\rangle m\delta_{m+n,0}=0.
\end{equation}
By (\ref{liealgops}), this means that 
\begin{equation}
 [Y(x_\alpha(-1)\mathbf{1},x_1),Y(x_\alpha(-1)\mathbf{1},x_2)]=0.
\end{equation}
Thus for any $k\geq 0$, the product $Y(x_\alpha(-1)\mathbf{1},x)^k$ is
well-defined and equals the normal-ordered product $\nordcirc
Y(x_{\alpha}(-1)\mathbf{1},x)^k\nordcirc $ (see \cite{LL} (3.8.4)). Then we can
apply Proposition 3.10.2 in \cite{LL} to conclude that
\begin{eqnarray}\label{partop}
 Y\left(\frac{x_{\alpha}(-1)^{k}}{k!}\mathbf{1},x\right)& = 
&\frac{Y(x_\alpha(-1)\mathbf{1},x)^k}{k!}\nonumber\\
& = & \frac{1}{k!}\left( \sum_{n_1\in\mathbb{Z}} x_\alpha(n_1) x^{-n_1-1}\right)
\cdots\left( \sum_{n_k\in\mathbb{Z}} x_\alpha(n_k) x^{-n_k-1}\right) \nonumber\\
& = &\frac{1}{k!}\sum_{l\in\mathbb{Z}}\left( \sum_{n_1+\cdots +n_k=l}
x_\alpha(n_1)\cdots x_\alpha(n_k)\right) x^{-l-k}.\\\nonumber
\end{eqnarray}
Consider the coefficient of $x^{-l-k}$ in (\ref{partop}) for any
$l\in\mathbb{Z}$, that is,
\begin{equation}\label{partop2}
 \frac{1}{k!}\sum_{n_1+\ldots +n_k=l} x_\alpha(n_1)\cdots x_\alpha(n_k).
\end{equation}
Since all the $x_\alpha(n)$ operators commute with each other, for any
$\sigma\in S_n$, $x_\alpha(n_1)\cdots x_\alpha(n_k)=x_\alpha(n_{\sigma
(1)})\cdots x_\alpha(n_{\sigma (k)})$. Thus we can collect some of the terms in
(\ref{partop2}).

Take any ``partition'' of $l$ into exactly $k$ parts, where parts are 
allowed to be negative or zero, as well as positive. Suppose the distinct parts
are $n_1,\ldots 
,n_m$ where $n_j$ occurs $i_j$ times, that is, $n_j\in\mathbb{Z}$, $n_1 
i_1+\ldots +n_m i_m=l$, and $i_1+\ldots +i_m=k$. The terms in the sum in 
(\ref{partop2}) corresponding to this partition are $x_\alpha(n_1)^{i_1}\cdots 
x_\alpha(n_m)^{i_m}$ and all permutations. The number 
of distinct permutations is
\begin{equation}
 \binom{k}{i_1}\binom{k-i_1}{i_2}\cdots\binom{k-i_1-\ldots -i_{m-1}}{i_m}
\end{equation}
(note that $\binom{k-i_1-\ldots -i_{m-1}}{i_m}=\binom{i_m}{i_m}=1$). This 
equals
\begin{equation}
 \frac{k!}{(k-i_1)!\,i_1!}\frac{(k-i_1)!}{(k-i_1-i_2)!\,i_2!}\cdots
1=\frac{k!}{i_1!\,i_2!\cdots i_m!}.
\end{equation}
Hence the sum of all terms in (\ref{partop2}) corresponding to this partition is
\begin{equation}
 \frac{1}{k!}\frac{k!}{i_1!\,i_2!\cdots i_m!}x_\alpha(n_1)^{i_1}\cdots
x_\alpha(n_m)^{i_m}=\frac{x_\alpha(n_1)^{i_1}}{i_1!}\cdots\frac{x_\alpha(n_m)^{
i_m}}{i_m!}.
\end{equation}
Thus the coefficient of $x^{-l-k}$ in $Y(x_\alpha(-1)\mathbf{1},x)^k/k!$ is
\begin{equation}\label{partop3}
 \sum_{\substack{\mathrm{partitions}\, \mathrm{of}\, l\\\mathrm{with}\,k\,\mathrm{parts}}}
\frac{x_\alpha(n_1)^{i_1}}{i_1!}\cdots\frac{x_\alpha(n_m)^{i_m}}{i_m!}.
\end{equation}

Considering the case $U=\mathbb{C}\mathbf{1}$, it is clear from 
(\ref{affintspan2}), (\ref{affintspan}) and (\ref{partop3}) that the vertex 
subring generated by the $\frac{x_\alpha(-1)^k}{k!}\mathbf{1}$
is contained in $V_{\widehat{\mathfrak{g}}}(\ell,0)_\mathbb{Z}$. On the
other hand, we can use induction on $k$ to show that for any $m\in\mathbb{Z}$
and $\alpha$ a root of $\mathfrak{g}$, $x_\alpha(m)^k/k!$ preserves the vertex
subring generated by the $\frac{x_\alpha(-1)^k}{k!}\mathbf{1}$. This is
trivially true for $k=0$. If $k>0$, take $l=mk$ in (\ref{partop3}) to conclude
that
\begin{equation}
 \frac{x_\alpha(m)^k}{k!}=\left(
\frac{x_\alpha(-1)^k}{k!}\mathbf{1}\right)_{mk+k-1}-\sum_{\mathrm{partitions}
\neq
(m,\ldots ,m)}
\frac{x_\alpha(n_1)^{i_1}}{i_1!}\cdots\frac{x_\alpha(n_m)^{i_m}}{i_m!}.
\end{equation}
Since each $i_j<k$ on the right side, by induction every term on the right side
preserves the vertex subring, so $x_\alpha(m)^k/k!$ does as well. Since
$\mathbf{1}$ is in any vertex subring, (\ref{affintspan2}) now implies that
$V_{\widehat{\mathfrak{g}}}(\ell,0)_\mathbb{Z}$ is contained in the
vertex subring generated by the $\frac{x_\alpha(-1)^k}{k!}\mathbf{1}.$
In the same way, it follows that for any $U$, 
$V_{\widehat{\mathfrak{g}}}(\ell,U)_\mathbb{Z}$ and 
$L_{\widehat{\mathfrak{g}}}(\ell,U)_\mathbb{Z}$ are the 
$V_{\widehat{\mathfrak{g}}}(\ell,0)_\mathbb{Z}$-modules generated by 
$U_\mathbb{Z}$
\end{proof}

\begin{corol}
 The integral form $L_{\widehat{\mathfrak{g}}}(\ell,0)_\mathbb{Z}$ of 
$L_{\widehat{\mathfrak{g}}}(\ell,0)$ is the integral form of 
$L_{\widehat{\mathfrak{g}}}(\ell,0)$ as a vertex algebra generated by the 
vectors $\frac{x_\alpha(-1)^k}{k!}\mathbf{1}$ where $\alpha$ is a root and 
$k\geq 0$.
\end{corol}
\begin{proof}
 We know from the proof of Theorem \ref{affalgzform} that 
$L_{\widehat{\mathfrak{g}}}(\ell,0)_\mathbb{Z}$ is spanned by 
coefficients of products of the form in (\ref{affintspan}), but by Proposition 
\ref{zgen}, this is precisely the vertex subring of 
$L_{\widehat{\mathfrak{g}}}(\ell,0)$ generated by the 
$\frac{x_\alpha(-1)^k}{k!}\mathbf{1}$. Note that although the $Y$ in 
(\ref{affintspan}) is the vertex operator for 
$V_{\widehat{\mathfrak{g}}}(\ell,0)$ acting on 
$L_{\widehat{\mathfrak{g}}}(\ell,0)$, the vertex operator for 
$L_{\widehat{\mathfrak{g}}}(\ell,0)$ acting on itself is defined the same way.
\end{proof}

\begin{corol}\label{affcorol}
 If $\mathfrak{g}$ is of type $A$, $D$, or $E$ and $\ell$ is a positive integer,
the integral form $L_{\widehat{\mathfrak{g}}}(\ell,0)_\mathbb{Z}$ of
$L_{\widehat{\mathfrak{g}}}(\ell,0)$ is generated by the vectors
$\frac{x_\alpha(-1)^k}{k!}\mathbf{1}$ where $\alpha$ is a root of $\mathfrak{g}$
and $0\leq k\leq\ell$.
\end{corol}
\begin{proof}
 This follows from the well-known fact that for any long root $\alpha$,
$x_\alpha(-1)^{\ell+1}\cdot v_0=0$, where $v_0$ is a highest weight vector of a
standard level $\ell\,$ $\widehat{\mathfrak{g}}$-module (see for example
Proposition 6.6.4 in \cite{LL}).
\end{proof}
\begin{rema}
 The integral form $L_{\widehat{\mathfrak{g}}}(\ell,0)_\mathbb{Z}$ as in 
Corollary \ref{affcorol} is generally larger than the one constructed using the 
integral form of $U(\widehat{\mathfrak{g}})$ discussed in Remark 
\ref{PBWintform}, although they coincide when $\mathfrak{g}$ is a 
finite-dimensional simply-laced simple Lie algebra and $\ell=1$.
\end{rema}

\section{Integral forms for lattice vertex algebras}

We recall the construction of conformal vertex algebras from even lattices. 
Suppose $L$ is a non-degenerate even lattice with form $\left\langle\cdot 
,\cdot \right\rangle $. Consider also the space 
$\mathfrak{h}=\mathbb{C}\otimes_\mathbb{Z} L$, an abelian Lie algebra with a 
(trivially) invariant form. Thus we can form the Heisenberg vertex operator 
algebra $V_{\widehat{\mathfrak{h}}}(1,0)$, which is linearly isomorphic to 
$S(\widehat{\mathfrak{h}}_-)$, the symmetric algebra on 
$\widehat{\mathfrak{h}}_-$. We can also construct the twisted group algebra 
$\mathbb{C}\lbrace L\rbrace$ as follows: given a positive integer $s$, take 
a central extension of $L$ by the cyclic group of $s$ elements:
\begin{equation}
 1\rightarrow\left\langle \kappa\mid\kappa^s=1\right\rangle 
\rightarrow\widehat{L}\bar{\rightarrow} L\rightarrow 1.
\end{equation}
If $\omega_s$ is a primitive $s$th root of unity, let $\mathbb{C}_{\omega_s}$ 
denote the one-dimensional $\langle \kappa\rangle$-module on which $\kappa$ 
acts as $\omega_s$. Then the twisted group algebra is the induced 
$\widehat{L}$-module
\begin{equation}
 \mathbb{C}\lbrace L\rbrace 
=\mathbb{C}[\widehat{L}]\otimes_{\mathbb{C}[\langle\kappa\rangle]}\mathbb{C}
_{\omega_s}.
\end{equation}
It is linearly isomorphic to the group algebra $\mathbb{C}[L]$. We need to 
choose the integer $s$ and the central extension so that the commutator map
$c_0$, defined by the 
condition $ab=ba\kappa^{c_0(\bar{a},\bar{b})}$ for $a,b\in\widehat{L}$, 
satisfies the condition
\begin{equation}
 \omega_s^{c_0(\alpha,\beta)}=(-1)^{\left\langle \alpha,\beta\right\rangle} 
\end{equation}
for $\alpha,\beta\in L$. For instance, since $L$ is even, we can take $s=2$ and 
$c_0(\alpha,\beta)=\left\langle \alpha,\beta\right\rangle $ (mod $2$). (The
resulting vertex operator algebra does not depend up to isomorphism on the
choices of $s$ and the central extension; see \cite{LL}, Proposition 6.5.5.) The
$\widehat{L}$-module 
$\mathbb{C}\lbrace L\rbrace$ is also a module for 
$\mathfrak{h}=\mathfrak{h}\otimes t^0$: 
\begin{equation}
 h(0)(a\otimes 1)=\langle h,\bar{a}\rangle (a\otimes 1)
\end{equation}
for $h\in\mathfrak{h}$ and $a\in\widehat{L}$. Also, for $\alpha\in L$ and $x$ a 
formal variable, define 
a map $x^\alpha$ on $\mathbb{C}\lbrace L\rbrace$ by
\begin{equation}
x^\alpha (b\otimes 1)=(b\otimes 1) x^{\langle\alpha,\bar{b}\rangle}
\end{equation}
for $b\in\widehat{L}$.

We can now extend the vertex operator algebra structure on 
$S(\widehat{\mathfrak{h}}_-)$ to the larger space
\begin{equation}
 V_L=S(\widehat{\mathfrak{h}}_-)\otimes\mathbb{C}\lbrace L\rbrace .
\end{equation}
For $a\in\widehat{L}$, use $\iota(a)$ to denote the element $1\otimes 
(a\otimes 1)\in V_L$. As a vertex algebra, $V_L$ is generated by 
the $\iota(a)$, which have vertex operators
\begin{equation}\label{lattops}
 Y(\iota(a),x)=E^- (-\bar{a},x)E^+ (-\bar{a},x) a x^{\bar{a}}
\end{equation}
where
\begin{equation}
 E^\pm (\alpha, x)=\mathrm{exp}\left( \sum_{n\in\pm\mathbb{Z}_+} 
\frac{\alpha(n)}{n} x^{-n}\right) 
\end{equation}
for any $\alpha\in L$, and $a$ denotes the action of $a\in\widehat{L}$ on 
$\mathbb{C}\lbrace L\rbrace$. The conformal element of $V_L$ is the same as 
the conformal element in the Heisenberg algebra:
\begin{equation}
 \omega =\frac{1}{2}\sum_{i=1}^{\mathrm{dim}\;\mathfrak{h}}
\alpha_i(-1)\alpha_i'(-1)\mathbf{1},
\end{equation}
where $\lbrace\alpha_i\rbrace$ is a basis for $\mathfrak{h}$ and
$\lbrace\alpha_i'\rbrace $ is the corresponding dual basis with respect to
$\langle\cdot,\cdot\rangle$.
 If $L$ is positive definite, 
$V_L$ is a vertex operator algebra in the sense that the finiteness 
restrictions on the weight spaces hold. If $L$ is not positive definite, 
$V_L$ is still a strongly $L$-graded conformal vertex algebra, where for
$\alpha\in L$,
\begin{equation}
 V^\alpha =S(\widehat{\mathfrak{h}}_-)\otimes\iota(a),
\end{equation}
with $a\in\widehat{L}$ such that $\bar{a}=\alpha$.

To obtain modules for $V_L$, consider $L^\circ$, the dual lattice of $L$, and 
construct the space
\begin{equation}
 V_{L^\circ}=S(\widehat{\mathfrak{h}}_-)\otimes\mathbb{C}\lbrace 
L^\circ\rbrace
\end{equation}
in the same way we constructed $V_L$. In particular, we need a central extension
of $L^\circ$ by $\langle\kappa\,\vert\,\kappa^s=1\rangle$, $s$ an even integer,
with commutator map $c_0$, having the property that 
\begin{equation}\label{c0form}
 \omega_s^{c_0(\alpha,\beta)}=(-1)^{\langle\alpha,\beta\rangle}
\end{equation}
for $\alpha,\beta\in L$, where $\omega_s$ is the $s$th root of unity used to
construct $\mathbb{C}\lbrace L^\circ\rbrace$. Such a central extension always
exists (see Remark 6.4.12 in \cite{LL}).

 Then $V_{L^\circ}$ is a $V_L$-module with 
the same action of $\iota(a)$ as in (\ref{lattops}). For any $S\subseteq
L^\circ$, denote by $\mathbb{C}\lbrace S\rbrace$ the subspace of
$\mathbb{C}\lbrace L^\circ\rbrace$ spanned by elements of the form $a\otimes 1$
where $\bar{a}\in S$. Then the spaces
\begin{equation}
 V_{\beta +L}=S(\widehat{\mathfrak{h}}_-)\otimes\mathbb{C}\lbrace 
\beta+L\rbrace ,
\end{equation}
where $\beta$ runs over coset representatives of $L^\circ /L$, exhaust the 
irreducible $V_L$-modules up to equivalence.

The vertex algebra $V_L$ (repectively, the module $V_{L^\circ}$) is spanned by
vectors of the form form
\begin{equation}
 \alpha_1(-n_1)\cdots \alpha_k(-n_k)\iota(b)
\end{equation}
where $\alpha_i\in L$, $n_i\in\mathbb{Z}_+$, and $\bar{b}\in L$ (respectively,
$\bar{b}\in L^\circ$). Such a vector has conformal weight
\begin{equation}
 n_1+\ldots +n_k+\frac{\langle \bar{b},\bar{b}\rangle}{2}\in\mathbb{Q}.
\end{equation}
The modules $V_{L^\circ}$ and $V_{\beta+L}$ where $\beta\in L^\circ$ are modules
for $V_L$ as strongly $L$-graded conformal vertex algebra since they are graded
by $L^\circ$:
\begin{equation}
 V_{L^\circ}=\coprod_{\gamma\in L^\circ} V^\gamma
\end{equation}
where $V^\gamma=S(\widehat{\mathfrak{h}}_-)\otimes\iota(c)$, $c\in L^\circ$ such
that $\bar{c}=\gamma$, and $V_{\beta+L}$ is analogously graded.

In order to obtain integral structure in $V_L$ and its modules, we first need to
show that we can choose the central extension so that we have integral structure
in the twisted group algebra $\mathbb{C}\lbrace L\rbrace$. Given a central
extension $\widehat{L^\circ}$ of $L^\circ$, we can choose a section
$L^\circ\rightarrow \widehat{L^\circ}$, denoted $\beta\mapsto e_\beta$ for
$\beta\in L^\circ$. We define a $2$-cocycle $\varepsilon_0: L^\circ\times
L^\circ\rightarrow\mathbb{Z}/s\mathbb{Z}$ by 
\begin{equation}
 e_\alpha e_\beta =e_{\alpha+\beta}\kappa^{\varepsilon_0(\alpha,\beta)}.
\end{equation}
Conversely, given a $2$-cocycle (such as a bilinear map) $\varepsilon_0:
L^\circ\times L^\circ\rightarrow\mathbb{Z}/s\mathbb{Z}$ we can define a central
extension of $L^\circ$ by $\langle\kappa\,\vert\,\kappa^s=1\rangle$ (see
Proposition 5.1.2 in \cite{FLM}). Given an $\varepsilon_0$, define
$\varepsilon(\alpha,\beta)=\omega_s^{\varepsilon_0(\alpha,\beta)}$. The
following lemma is an adjustment of Remark 6.4.12 in \cite{LL} (see also 
Remark 12.17 in \cite{DL}):
\begin{lemma}\label{centext}
 There exists a central extension $\widehat{L^\circ}$ satisfying (\ref{c0form}) 
and a section
$L^\circ\rightarrow\widehat{L^\circ}$ such that $\varepsilon(\alpha,\beta)=\pm
1$ for any $\alpha,\beta\in L$.
\end{lemma}
\begin{proof}
It is possible to choose a base $\lbrace\alpha_1,\ldots,\alpha_l\rbrace$ of
$L^\circ$ so that $\lbrace n_1 \alpha_1,\ldots,n_l\alpha_l\rbrace$, where the
$n_i$ are positive integers, forms a base for $L$. Since
$\langle\cdot,\cdot\rangle$ is $\mathbb{Q}$-valued on $L^\circ$, there is a
positive even integer $s$ such that
$\frac{s}{2}\langle\alpha,\beta\rangle\in\mathbb{Z}$ for any $\alpha,\beta\in
L^\circ.$ Then we have the bilinear $2$-cocycle $\varepsilon_0:L^\circ\times
L^\circ\rightarrow\mathbb{Z}/s\mathbb{Z}$ determined by its values on the base:
\begin{equation}
 \varepsilon_0(\alpha_i,\alpha_j)=\left\lbrace
\begin{array}{ccc}  
\frac{s}{2}\langle\alpha_i,\alpha_j\rangle +s\mathbb{Z} & \mathrm{if} & i<j\\ 
0 & \mathrm{if} & i\geq j\\ 
\end{array}.\right.
\end{equation}
Using bilinearity, we have for $i<j$
\begin{equation}
 \varepsilon_0(n_i\alpha_i,n_j\alpha_j)=n_i
n_j\varepsilon_0(\alpha_i,\alpha_j)=n_i
n_j\frac{s}{2}\langle\alpha_i,\alpha_j\rangle +s\mathbb{Z}=\frac{s}{2}\langle
n_i\alpha_i,n_j\alpha_j\rangle+s\mathbb{Z}.
\end{equation}
Since $\langle\cdot,\cdot\rangle$ is $\mathbb{Z}$-valued and bilinear on $L$, it
follows that $\varepsilon(\alpha,\beta)=\pm 1 $ for $\alpha,\beta\in L$. 

The $2$-cocycle $\varepsilon_0$ corresponds to a section of a central extension
$\widehat{L^\circ}$ of $L^\circ$. The corresponding commutator map $c_0$ is
given by
\begin{equation}
c_0(\alpha,\beta)=\varepsilon_0(\alpha,\beta)-\varepsilon_0(\beta,\alpha)
\end{equation}
for $\alpha,\beta\in L^\circ$ since
\begin{equation}
 e_\alpha e_\beta=e_{\alpha+\beta}\kappa^{\varepsilon_0(\alpha,\beta)}=e_\beta
e_\alpha\kappa^{-\varepsilon_0(\beta,\alpha)}\kappa^{\varepsilon_0(\alpha,\beta)
}.
\end{equation}
Thus the commutator map is an alternating $\mathbb{Z}$-bilinear form
$L^\circ\times L^\circ\rightarrow\mathbb{Z}/s\mathbb{Z}$. Since
\begin{equation}
 c_0(n_i\alpha_i,n_j\alpha_j)=\frac{s}{2}\langle
n_i\alpha_i,n_j\alpha_j\rangle+s\mathbb{Z}
\end{equation}
for $i<j$ and since
\begin{equation}
 (\alpha,\beta)\mapsto\frac{s}{2}\langle\alpha,\beta\rangle+s\mathbb{Z}
\end{equation}
is also an alternating bilinear form on $L$ (since $L$ is even), it follows that
\begin{equation}
 c_0(\alpha,\beta)=\frac{s}{2}\langle\alpha,\beta\rangle+s\mathbb{Z}
\end{equation}
for all $\alpha,\beta\in L$. Thus
\begin{equation}
 \omega_s^{c_0(\alpha,\beta)}=\omega_s^{s\langle\alpha,\beta\rangle/2}=(-1)^{
\langle\alpha,\beta\rangle}
\end{equation}
for $\alpha,\beta\in L$, as required.
\end{proof}

Throughout the rest of this section we will use the central extension and
section of Lemma \ref{centext}. Moreover, we will assume that $e_0=1$, so that
$\iota(e_0)=\mathbf{1}$. Thus $V_L$ is generated as vertex algebra by the
$\iota(e_\alpha)$ for $\alpha\in L$. This motivates the following definition:
\begin{defi}\label{lattdef}
 Define $V_{L,\mathbb{Z}}$ to be the vertex subring of $V_L$ 
generated by the $\iota(e_\alpha)$ for $\alpha\in L$. Moreover, for any
$\beta\in L^\circ$, set $V_{\beta+L,\mathbb{Z}} $ equal to the $V_{L,\mathbb{Z}}
$ submodule of $V_{\beta+L}$ generated by $\iota(e_\beta)$. In particular,
$V_{L,\mathbb{Z}}$ itself is the $V_{L,\mathbb{Z}}$ submodule of $V_L$ generated
by $\mathbf{1}$.
\end{defi}
Then we have the following theorem (the algebra part is originally due to 
Borcherds \cite{B} and has also been proved using a different method in \cite{P}
and \cite{DG}):

\begin{theo}\label{lattform}
The vertex subring $V_{L,\mathbb{Z}}$ is an integral form of $V_L$, and for any
$\beta\in L^\circ$, $V_{\beta+L,\mathbb{Z}}$ is an integral form of
$V_{\beta+L}$.
\end{theo}
\begin{proof}
Since $V_{L,\mathbb{Z}}$ is closed under vertex algebra products by definition,
we just need to show that for any $\beta\in L^\circ$, $V_{\beta+L,\mathbb{Z}}$
is an integral form of the vector space $V_{\beta+L}$ and that
$V_{\beta+L,\mathbb{Z}}$ is compatible with the
$L^\circ\times\mathbb{Q}$-gradation of $V_{\beta+L}$, as in
(\ref{compatibility2}). By Proposition \ref{zgen},  $V_{\beta+L,\mathbb{Z}}$ is
the $\mathbb{Z}$-span of coefficients of products of the form
\begin{equation}\label{spanners}
 Y(\iota(e_{\alpha_1}),x_1)\ldots Y(\iota(e_{\alpha_k}),x_k)\iota(e_\beta)
\end{equation}
where $\alpha_i\in L$.  Since $\iota(e_\beta)$ and $\iota(e_\alpha)$ for
$\alpha\in L$ are homogeneous in the $L^\circ\times\mathbb{Q}$-gradation of
$V_{L^\circ}$, it follows from (\ref{agradingops}) that coefficients of products
as in (\ref{spanners}) are doubly homogeneous. Hence $V_{\beta+L,\mathbb{Z}}$ is
compatible with the $L^\circ\times\mathbb{Q}$-gradation:
\begin{equation}\label{lattcompatibility}
 V_{\beta+L,\mathbb{Z}}=\coprod_{\gamma\in\beta+L,\,n\in\mathbb{Q}}
V^\gamma_n\cap V_{\beta+L,\mathbb{Z}}.
\end{equation}

To show that $V_{\beta+L,\mathbb{Z}}$ is an integral form of the vector space
$V_{\beta+L}$, it is enough to show that for any $\gamma\in\beta+L$,
$V^\gamma_n\cap V_{\beta+L,\mathbb{Z}}$ is a lattice in $V^\gamma_n$ whose rank
is the dimension of $V^\gamma_n$. Since the vectors $\iota(e_\alpha)$ for
$\alpha\in L$ generate $V_L$ as a vertex algebra and since $\iota(e_\beta)$
generates $V_{\beta+L}$ as a $V_L$-module, $V_{\beta+L,\mathbb{Z}}$ spans
$V_{\beta+L}$ over $\mathbb{C}$, and thus by (\ref{lattcompatibility}),
$V^\gamma_n\cap V_{\beta+L,\mathbb{Z}}$ spans $V^\gamma_n$ over $\mathbb{C}$ for
any $\gamma\in\beta+L$. Thus if $V^\gamma_n\cap V_{\beta+L,\mathbb{Z}}$ is a
lattice, its rank is at least the dimension of $V^\gamma_n$.

On the other hand, since $\varepsilon(\alpha,\beta)=\pm 1\in\mathbb{Q}$ for any
$\alpha,\beta\in L$, $V_L$ and $V_{\beta+L}$ have $\mathbb{Q}$-forms, namely,
the $\mathbb{Q}$-subalgebra $V_{L,\mathbb{Q}}\subseteq V_L$ generated by the
$\iota(e_\alpha)$ for $\alpha\in L$, and the $V_{L,\mathbb{Q}}$-submodule
$V_{\beta+L,\mathbb{Q}}$ generated by $\iota(e_\beta)$, respectively.  They are
spanned over $\mathbb{Q}$ by the vectors
\begin{equation}
 \alpha_1(-n_1)\cdots \alpha_k(-n_k)\iota(e_\gamma)
\end{equation}
 where $\alpha_i\in L$, $n_i\in\mathbb{Z}_+$, and $\gamma\in L$ or
$\gamma\in\beta+L$, respectively. Now, any set of vectors in $V^\gamma_n\cap
V_{\beta+L,\mathbb{Z}}$ which is linearly independent over $\mathbb{Z}$ is
linearly independent over $\mathbb{Q}$ since a dependence relation over
$\mathbb{Q}$ reduces to a dependence relation over $\mathbb{Z}$ by clearing
denominators. Thus, since $V_{L,\mathbb{Z}}\subseteq V_{L,\mathbb{Q}}$, any set
of vectors in $V^\gamma_n\cap V_{\beta+L,\mathbb{Z}}$ which is linearly
independent over $\mathbb{Z}$ is linearly independent over $\mathbb{C}$. This
means that if $V^\gamma_n\cap V_{\beta+L,\mathbb{Z}}$ is a lattice, its rank is
no more than the dimension of $V^\gamma_n$. 

Thus we are reduced to showing that for any $\gamma\in\beta+L$, $V^\gamma_n\cap
V_{\beta+L,\mathbb{Z}}$ is a lattice in $V^\gamma_n$, that is, it is spanned
over $\mathbb{Z}$ by a finite set. To show this, we use formula (8.4.22) in
\cite{FLM} to obtain
\begin{eqnarray}
 Y(\iota(e_{\alpha_1}),x_1)\cdots Y(\iota(e_{\alpha_k}),x_k)\iota(e_\beta) & = &
\nordcirc
Y(\iota(e_{\alpha_1}),x_1)\cdots
Y(\iota(e_{\alpha_k}),x_k)\nordcirc\iota(e_\beta)\cdot\nonumber\\
& & \prod_{1\leq i<j\leq k} (x_i-x_j)^{\langle
\alpha_i,\alpha_j\rangle }.\\\nonumber
\end{eqnarray}
where $\alpha_i\in L$. (See also formula (8.6.6) in 
\cite{FLM}; we use the normal ordering notation of \cite{FLM} here.)
Since $L$ is even and in particular integral, the binomial product expansions 
on the right side involve only integer coefficients. Hence coefficients of 
the non-normal ordered product are integral combinations of coefficients of 
the corresponding normal ordered product, and vice versa since the 
coefficients of
\begin{equation}
 \prod_{1\leq i<j\leq k} (x_i-x_j)^{-\langle 
\alpha_i,\alpha_j\rangle}
\end{equation}
are also integers. This shows that we can define $V_{\beta+L,\mathbb{Z}}$ as 
the $\mathbb{Z}$-span of coefficients in products of the form
\begin{equation}\label{nordprod}
 \nordcirc Y(\iota(e_{\alpha_1}),x_1)\cdots
Y(\iota(e_{\alpha_k}),x_k)\nordcirc\iota(e_\beta)
\end{equation}
where $\alpha_i\in L$. By the definition of normal 
ordering, (\ref{nordprod}) equals
\begin{equation}
 x_1^{\langle \alpha_1,\beta\rangle }\cdots x_k^{\langle 
\alpha_k,\beta\rangle } E^-(-\alpha_1,x_1)\cdots 
E^-(-\alpha_k,x_k)\iota(e_{\alpha_1}\cdots e_{\alpha_k} e_\beta).
\end{equation}
Thus, because $\varepsilon(\alpha,\beta)=\pm 1$ for any $\alpha,\beta\in L$,
$V_{\beta+L,\mathbb{Z}}$ is the $\mathbb{Z}$-span of coefficients 
in products of the form
\begin{equation}\label{lattbasis}
 E^{-}(-\alpha_1,x_1)\cdots E^{-}(-\alpha_k,x_k)\iota (e_\gamma)
\end{equation}
where $\alpha_i\in L$ and $\gamma\in \beta+L$. In fact, if $\left\lbrace 
\alpha^{(1)},\ldots,\alpha^{(l)}\right\rbrace $ is a base for $L$ (or any
finite 
spanning set), we may take the $\alpha_i$ in (\ref{lattbasis}) to come 
from $\left\lbrace \pm\alpha^{(1)},\ldots,\pm\alpha^{(l)}\right\rbrace $, since
if 
$\alpha =\sum_{i=1}^l n_i \alpha^{(i)}\in L$, then by properties of
exponentials,
\begin{equation}
 E^-(-\alpha,x)=\prod_{i=1}^l E^-(-\alpha^{(i)},x)^{n_i},
\end{equation}
where if $n_i$ is negative,
$E^-(-\alpha^{(i)},x)^{n_i}=E^-(\alpha^{(i)},x)^{-n_i}.$

Recall that for $\alpha\in L$, 
\begin{equation}
 E^{-}(-\alpha ,x)=\mathrm{exp}\left( 
\sum_{n>0}\frac{\alpha(-n)}{n} x^{n}\right),
\end{equation}
 so the coefficient of $x^m$ in this operator increases weight by $m$. Hence the
coefficient of any monomial of total degree $m$ in (\ref{lattbasis}) is in
$V^{\gamma}_{m+\left\langle \gamma, \gamma\right\rangle /2}$. The coefficients
of such monomials, with $\gamma $ fixed, for which $m+\left\langle \gamma,
\gamma\right\rangle /2=n\,$ span $V_{\beta+L,\mathbb{Z}}\cap V^{\gamma}_n$.
Since 
$\left\lbrace\pm\alpha^{(1)},\ldots ,\pm\alpha^{(l)}\right\rbrace $ is a finite
set, there are only a finite number of ways of obtaining coefficients of
products 
of the form  (\ref{lattbasis}) that lie in $V^{\gamma}_n$. This shows that 
$V_{\mathbb{Z}}\cap V^{\gamma}_n$ is finitely generated for any 
$n\in\mathbb{Q},\gamma\in \beta+L$, completing the proof.
\end{proof}

\begin{rema}
Borcherds' definition of $V_{L,\mathbb{Z}}$ in \cite{B} does not use the vertex
algebra structure of $V_L$; note that $V_L$ is also an associative algebra with
product determined by
\begin{eqnarray}
(\alpha_1(-m_1)\cdots\alpha_j(-m_j)\iota(a))(\beta_1(-n_1)\cdots\beta_k(-n_k)
\iota(b))=&&\nonumber\\
\alpha_1(-m_1)\cdots\alpha_j(-m_j)\beta_1(-n_1)\cdots\beta_k(-n_k)\iota(ab),
&&\nonumber\\\nonumber
\end{eqnarray}
where $\alpha_i,\beta_i,\bar{a},\bar{b}\in L$, $m_i,n_i>0$. Thus $V_L$ is
generated as an associative algebra by the elements 
$\iota(a)$ for $\bar{a}\in L$ and $\alpha(-n)\mathbf{1}$ where $\alpha\in L$ 
and $n>0$. There is a derivation $D$ of this associative algebra structure 
defined on generators by $D\iota(a)=\bar{a}(-1)\iota(a)$ and  
$D\alpha(-n)\mathbf{1}=n\alpha(-n-1)\mathbf{1}$. This is precisely the action 
of $L(-1)$ on these elements, and in fact $D=L(-1)$. In \cite{B}, 
$V_{L,\mathbb{Z}}$ is defined to be the smallest associative subring of $V_L$ 
containing each $\iota(e_\alpha)$ and invariant under $D^i/i!$ for $i\geq 0$. It
is 
claimed that $V_{L,\mathbb{Z}}$ is then generated as associative ring by the 
$\iota(e_\alpha)$ and the coefficients of $E^-(-\alpha,x)\mathbf{1}$, that is, 
$V_{L,\mathbb{Z}}$ is the $\mathbb{Z}$-span of coefficients of products of 
the form (\ref{lattbasis}) (where $\gamma$ is now in $L$). From the proof of 
Theorem \ref{lattform}, we know that such vectors span $V_{L,\mathbb{Z}}$ as 
we have defined it here. Borcherds' claim has been proven in \cite{P}, but we
simplify the proof here:
\end{rema}
\begin{propo}
The definition of $V_{L,\mathbb{Z}}$ in \cite{B} agrees with Definition
\ref{lattdef}.
\end{propo}
\begin{proof}
Let $V_{L,\mathbb{Z}}^*$ denote the structure defined in \cite{B}, 
and $V_{L,\mathbb{Z}}$ the structure of Definition \ref{lattdef}. First, we 
show $V_{L,\mathbb{Z}}\subseteq V_{L,\mathbb{Z}}^*$. Since 
$V_{L,\mathbb{Z}}$ is the $\mathbb{Z}$-span of coefficients of products of 
the form (\ref{lattbasis}) and $V_{L,\mathbb{Z}}^*$ is an associative 
subring, it is enough to show that each $\iota(e_\alpha)$ for $\alpha\in L$ and
the 
coefficients of each $E^-(-\alpha,x)\mathbf{1}$ for $\alpha\in L$ are in 
$V_{L,\mathbb{Z}}^*$. Now, each $\iota(e_\alpha)\in 
V_{L,\mathbb{Z}}^*$ by definition; also $V_{L,\mathbb{Z}}^*$ 
is closed under $L(-1)^i/i!$ for each $i\geq 0$. Thus for any $\alpha\in L$, 
$V_{L,\mathbb{Z}}^*$ must contain the coefficients of
\begin{equation}
 e^{L(-1)x}\iota(e_\alpha)=Y(\iota(e_\alpha),x)\mathbf{1}=E^-(-\alpha,
x)\iota(e_\alpha).
\end{equation}
Recall that in any conformal vertex algebra, $e^{L(-1)x}v=Y(v,x)\mathbf{1}$ 
for any $v$ (formulas (3.1.29) and (3.1.67) in \cite{LL}). Since 
$V_{L,\mathbb{Z}}^*$ is an associative subring, it contains the 
coefficients of
\begin{equation}
 (E^-(-\alpha,x)\iota(e_\alpha))(\iota(e_{-\alpha}))=\pm
E^-(-\alpha,x)\mathbf{1},
\end{equation}
since $\varepsilon(\alpha,-\alpha)=\pm 1$.

On the other hand, $V_{L,\mathbb{Z}}$ is an associative subring of $V_L$ (the 
associative product of any two coefficients of products of the form 
(\ref{lattbasis}) is again such a coefficient). Also, $V_{L,\mathbb{Z}}$ is 
invariant under each $L(-1)^i/i!$ since it is closed under vertex operators 
and $e^{L(-1)x}v=Y(v,x)\mathbf{1}$ for $v\in V_{L,\mathbb{Z}}$. Thus  
$V_{L,\mathbb{Z}}^*\subseteq V_{L,\mathbb{Z}}$, and 
$V_{L,\mathbb{Z}}=V_{L,\mathbb{Z}}^*$.
\end{proof}

\begin{rema}
 If $L$ is the root lattice of a finite-dimensional simple Lie algebra
$\mathfrak{g}$ of type $A$, $D$, or $E$, the lattice vertex operator algebra
$V_L$ is isomorphic to the level $1$ affine Lie algebra vertex operator algebra
$L_{\widehat{\mathfrak{g}}} (1,0)$. The isomorphism is determined by
\begin{equation}
 \iota(e_\alpha)\mapsto\pm x_{\alpha} (-1)\mathbf{1},
\end{equation}
for $\alpha$ a root of $\mathfrak{g}$ and $x_{\alpha}$ the corresponding
root vector in a Chevalley basis for $\mathfrak{g}$. (For the proof of this
result see \cite{FLM} and \cite{DL}.) From the definitions, it is clear the
integral forms $V_{L,\mathbb{Z}}$ and
$L_{\widehat{\mathfrak{g}}}(1,0)_\mathbb{Z}$ correspond under this isomorphism.
\end{rema}

The following result on a $\mathbb{Z}$-base for $V_{\beta+L,\mathbb{Z}}$ has
been proved for the algebra case $\beta=0$ in \cite{DG}, but
since we will need it later, we include the proof for completeness; assume now
that $\left\lbrace \alpha^{(1)},\ldots\alpha^{(l)}\right\rbrace $ is a base for
$L$:
\begin{propo}\label{lattbasisprop}
 The distinct coefficients of monomials in products as in (\ref{lattbasis}) 
form a basis for $V_{\beta+L,\mathbb{Z}}$, where the $\alpha_i$ come from 
$\left\lbrace \alpha^{(1)},\ldots ,\alpha^{(l)}\right\rbrace $ and $\gamma$ is
any 
element of $\beta+L$.
\end{propo}
\begin{proof}
 The proof of Theorem \ref{lattform} shows that the coefficients of monomials 
in (\ref{lattbasis}) span $V_{\beta+L,\mathbb{Z}}$ when the $\alpha_i$ come
from 
$\left\lbrace\pm\alpha^{(1)},\ldots ,\pm\alpha^{(l)}\right\rbrace $. However,
recall 
that $E^-(\alpha^{(i)},x)=E^-(-\alpha^{(i)},x)^{-1}$. If we expand
\begin{equation}
E^-(-\alpha^{(i)},x)=1+\sum_{j\geq 1} y_{ij}x^j,
\end{equation}
where $y_{ij}$ is a polynomial in the $\alpha^{(i)}(-k)$, then
\begin{equation}
 E^-(\alpha^{(i)},x)=\frac{1}{1+\sum_{j\geq 1} y_{ij}x^j} =\sum_{n\geq 0} \left(
-\sum_{j\geq 1} y_{ij} x^j\right) ^n.
\end{equation}
Thus the coefficients of $E^-(\alpha^{(i)},x)$ are polynomials in the 
coefficients of  $E^-(-\alpha^{(i)},x)$. This shows that the coefficients of 
monomials in (\ref{lattbasis}) span $V_{L^\circ,\mathbb{Z}}$ when the
$\alpha_i\in\left\lbrace \alpha^{(1)},\ldots ,\alpha^{(l)}\right\rbrace $.

We also need to show that the indicated coefficients are linearly independent 
(over $\mathbb{Z}$). In fact, they are linearly independent over 
$\mathbb{C}$, and to show this, it is sufficient to show that the polynomials 
$y_{ij}$ are algebraically independent in $S(\widehat{\mathfrak{h}}_-)$. 
Since 
\begin{equation}
 E^-(-\alpha^{(i)},x)=\mathrm{exp}\left( \sum_{n<0} \frac{-\alpha^{(i)}(n)}{n}
x^{-n}\right) =\mathrm{exp}\left( \sum_{n>0} \frac{\alpha^{(i)}(-n)}{n}
x^n\right), 
\end{equation}
$y_{ij}=\alpha^{(i)}(-j)/j+F_{ij}$, where $F_{ij}$ is a polynomial in the 
$\alpha^{(i)}(-k)$ with degree greater than $1$ and with $k<j$.

Now suppose there is a relation
\begin{equation}
 F=\sum c_{i_1\ldots i_k ;j_1\ldots j_k} y_{i_1 j_1}\cdots y_{i_k j_k} =0,
\end{equation}
where all coefficients $c_{i_1\ldots i_k ;j_1\ldots j_k}$ are non-zero. If 
$k_{min}$ is the smallest degree of any monomial in the $y_{ij}$ in $F$, then 
the term of minimal degree in the $\alpha^{(i)}(-j)$ is
\begin{equation}
 \sum_{k=k_{min}} \frac{c_{i_1\ldots i_k ;j_1\ldots j_k}}{j_1\cdots j_k}
\alpha^{(i_1)}(-j_1)\cdots\alpha^{(i_k)}(-j_k).
\end{equation}
 Since the $\alpha^{(i)}(-j)$ are algebraically independent, this sum must equal
$0$; 
but then each $c_{i_1\ldots i_k ;j_1\ldots j_k}=0$ for $k$ minimal as well. 
This contradiction shows that no nontrivial relation $F(\left\lbrace 
y_{ij}\right\rbrace )=0$ exists, so the $y_{ij}$ are algebraically 
independent.
\end{proof}

\begin{rema}
 We can express this basis for $V_{\beta+L,\mathbb{Z}}$ as the elements
\begin{equation}
 y_{i_1 j_1}\cdots y_{i_k j_k}\iota(e_\gamma)
\end{equation}
where $k\geq 0$, $1\leq i_1\leq\ldots\leq i_k\leq l$, $j_m\leq j_{m+1}$ if 
$i_m=i_{m+1}$, and $\gamma\in \beta+L$.
\end{rema}

\section{The conformal vector in an integral form}

Suppose $A$ is an abelian group and $V$ is a strongly $A$-graded conformal
vertex algebra with conformal vector $\omega$ and
central charge $c\in\mathbb{C}$. Then we have the following result on when an
integral form $V_\mathbb{Z}$ of $V$ can contain $\omega$:
\begin{propo}\label{omegainintform1}
 If $V_\mathbb{Z}$ contains $k\omega$ where $k\in\mathbb{C}$, then $k^2 c\in
2\mathbb{Z}$.
\end{propo}
\begin{proof}
 If $k\omega\in V_\mathbb{Z}$, then $V_\mathbb{Z}$ must also contain
\begin{equation}
 (kL(2))(k(L(-2))\mathbf{1}=k^2 L(-2)L(2)\mathbf{1}+4 k^2 L(0)\mathbf{1}+k^2
\frac{c(2^3-2)}{12}
\mathbf{1}=\frac{k^2 c}{2}\mathbf{1}.
\end{equation}
By Proposition \ref{spanofvacuum}, we must have $k^2 c\in 2\mathbb{Z}$.
\end{proof}
In particular, the central charge of $V$ must be an even integer if $\omega$ is
in any integral form of $V$. Now we prove a partial converse to Proposition
\ref{omegainintform1}. Recall that $v\in V$ is called a lowest weight
vector for the Virasoro algebra if it is an $L(0)$-eigenvector and $L(n)v=0$ for
$n> 0$.
\begin{theo}\label{omegatheorem}
 Suppose $V_\mathbb{Z}$ is an integral form of $V$ generated by doubly
homogeneous lowest weight
vectors $\lbrace v^{(j)}\rbrace$ for the Virasoro algebra. If $k\in\mathbb{Z}$
is such that $k^2 c\in 2\mathbb{Z}$ and $k\omega\in
V_\mathbb{Q}=\mathbb{Q}\otimes_\mathbb{Z} V_\mathbb{Z}$, then $V_\mathbb{Z}$ can
be extended to an
integral form of $V$ containing $k\omega$.
\end{theo}
\begin{proof}
 We shall show that the vertex subring $V^*_\mathbb{Z}$ of $V$ generated by
$V_\mathbb{Z}$ and $k\omega$ is an integral form of $V$. By Proposition
\ref{zgen}, $V_\mathbb{Z}^*$ is spanned over $\mathbb{Z}$ by coefficients of
products as in (\ref{zspan}) where the $u_i$ are either $v^{(j)}$ or $k\omega $.
Since the $v_j$ and $k\omega$ are homogeneous in the
$A\times\mathbb{Z}$-gradation of $V$,
\begin{equation}
 V_\mathbb{Z}^*=\coprod_{\alpha\in A,\,n\in\mathbb{Z}} V_\mathbb{Z}^*\cap
V^\alpha_n.
\end{equation}
Since $k\omega\in V_{\mathbb{Q}}$ and $V_{\mathbb{Z}}^*$ contains 
$V_{\mathbb{Z}}$, which spans $V^{\alpha}_n$, it is enough to show that 
$V_{\mathbb{Z}}^*\cap V^{\alpha}_n$ is finitely generated as an abelian group, 
just as in the proof of Theorem \ref{lattform}.

\begin{lemma} \label{omegalemma}
For any $m,n\in\mathbb{Z}$ and lowest weight vector $v$, $\left[ 
L(m),v_n\right] $ is an integral linear combination of operators 
$v_k$ for $k\in\mathbb{Z}$.
\end{lemma}
\begin{proof}
By the commutator formula,
\begin{eqnarray}
 [ Y(\omega,  x_1)  ,  Y(  v,x_2)] & = &
\mathrm{Res}_{x_0}\,x_2^{-1}\delta\left( \frac{x_1-x_0}{x_2}\right)
Y(Y(\omega,x_0)v,x_2)\nonumber\\
 & = & \mathrm{Res}_{x_0}\,e^{-x_0\,\partial /\partial x_1}\left(
x_2^{-1}\delta\left( \frac{x_1}{x_2}\right)\right) \sum_{n\in\mathbb{Z}}
Y(L(n)v,x_2)x_0^{-n-2}\nonumber\\
& = & \sum_{n\geq -1} (-1)^{n+1} \left( \dfrac{\partial}{\partial
x_1}\right)^{n+1} \left( x_2^{-1}\delta\left( \frac{x_1}{x_2}\right)\right)
Y(L(n)v,x_2).\\\nonumber
\end{eqnarray}
Since $v$ is a lowest weight vector, and using the $L(-1)$-derivative property,
\begin{eqnarray}
 [ Y(\omega,  x_1)  ,  Y(  v,x_2)] & = & x_2^{-1}\delta\left(
\frac{x_1}{x_2}\right) \dfrac{d}{dx_2} Y(v,x_2)\nonumber\\
& & -(\mathrm{wt}\; v)\dfrac{\partial}{\partial x_1}\left( x_2^{-1}\delta\left(
\frac{x_1}{x_2}\right) \right)     Y(v,x_2) \\\nonumber
\end{eqnarray}
Since $\mathrm{wt}\;v\in\mathbb{Z}$ and all the 
coefficients in the delta function expressions are integers, we see that 
$[L(m),v_n]$ is an integral combination of operators $v_k$.
\end{proof}
\begin{corol} \label{omegacorol}
 For $m\geq -1$, $L(m)$ leaves $V_{\mathbb{Z}}$ invariant.
\end{corol}
\begin{proof}
By Lemma \ref{omegalemma}, an expression of the form 
$L(m)v^{(j_1)}_{n_1}\cdots v^{(j_k)}_{n_k}\mathbf{1}$ equals 
$v^{(j_1)}_{n_1}\cdots v^{(j_k)}_{n_k}L(m)\mathbf{1}$ plus
an 
integral linear combination of terms of the form 
$v^{(j_1)}_{m_1}\cdots v^{(j_k)}_{m_k}\mathbf{1}$. But since $m\geq 
-1$, $L(m)\mathbf{1}=0$, so we see that $L(m) V_{\mathbb{Z}}\subseteq 
V_{\mathbb{Z}}$.
\end{proof}
Continuing with the proof of Theorem \ref{omegatheorem}, we can use Lemma
\ref{omegalemma} to rewrite 
any product of operators of the form $u^{(1)}_{m_1}\cdots u^{(k)}_{m_k}$, 
where the $u^{(i)}$ equal either $v^{(j)}$ or $k\omega$, as an integral 
combination of such products in which all operators of the form $kL(m)$ 
appear on the left. That is, $V_{\mathbb{Z}}^*$ is the integral span of 
products of the form
\begin{equation}
 (kL(m_1))\cdots
(kL(m_j))v^{(j_1)}_{n_1}\cdots v^{(j_k)}_{n_k}\mathbf{1},
\end{equation}
where $m_i,n_i\in\mathbb{Z}$. We can now use the Virasoro algebra relations,
\begin{equation}\label{kvir}
  \left[ kL(m),kL(n)\right] =k(m-n)(kL(m+n))+
\frac{k^2 c(m^3-m)}{12} 1_{V},
\end{equation}
to rewrite $(kL(m_1))\cdots (kL(m_j))$ as an integral combination of such 
products for which $m_1\leq\ldots\leq m_j$.  (Note that since $k^2c\in 
2\mathbb{Z}$ and $\frac{m^3-m}{6}=\binom{m+1}{3}$,
$\frac{k^2 c(m^3-m)}{12}$ is always an integer.)

Thus, using Corollary \ref{omegacorol}, we see that $V_{\mathbb{Z}}^*$ is the
integral span of products of the form
\begin{equation}
(kL(-m_1))\cdots (kL(-m_j))v
\end{equation}
where $m_i\geq 2$ and $v\in V_{\mathbb{Z}}$. To see that $V_{\mathbb{Z}}^*\cap 
V^{\alpha}_n$ is finitely generated, suppose $v_1,\ldots v_l\in 
V_{\mathbb{Z}}$ span $V_{\mathbb{Z}}\cap\left( \coprod_{i=N_\alpha}^n
V^{\alpha}_i\right)$, where $V^\alpha_m =0$ for $m<N_\alpha$. Then 
$V_{\mathbb{Z}}^*\cap V^{\alpha}_n$ is spanned by some of the vectors of the 
form 
$(kL(-m_1))\cdots (kL(-m_j))v_k$ where $m_i\geq 2$ and $m_1+\ldots +m_j\leq 
n-N_\alpha$. Since there are finitely many 
such vectors, $V_{\mathbb{Z}}^*\cap V^{\alpha}_n$ is finitely generated as an 
abelian group.
\end{proof}
\begin{propo}\label{omegaprop}
 Suppose $V_\mathbb{Z}$ is the integral form $V_{L,\mathbb{Z}}$ of a lattice
vertex operator algebra 
$V_L$ or the integral form $V_{\widehat{\mathfrak{g}}}(\ell,0)_\mathbb{Z}$ or 
$L_{\widehat{\mathfrak{g}}}(\ell,0)_\mathbb{Z}$
of an affine Lie algebra vertex operator algebra
 associated to a finite dimensional simple
Lie algebra $\mathfrak{g}$, with $\ell\in\mathbb{Z}$. Then $\omega\in 
V_{\mathbb{Q}}$ and $V_\mathbb{Z}$ is generated by lowest weight vectors for the 
Virasoro algebra, so $V_\mathbb{Z}$ may be extended to an integral form 
$V_\mathbb{Z}^*$ containing $k\omega$ for any $k\in\mathbb{Z}$ such that 
$k^2c\in2\mathbb{Z}$.
\end{propo}
\begin{proof}
If $V=V_L$ and $\lbrace\alpha_1,\ldots,\alpha_l\rbrace$ is a base for $L$ with 
dual base $\lbrace\alpha_1',\ldots,\alpha_l'\rbrace$ for $L^\circ$, then 
 \begin{equation}
 \omega =\frac{1}{2}\sum_{i=1}^{l}
\alpha_i(-1)\alpha_i'(-1)\mathbf{1}.
\end{equation}
 Since $\langle\cdot,\cdot\rangle $ is integral on $L$, $\alpha_i'\in 
\mathbb{Q}\otimes_\mathbb{Z} L$ for any $1\leq i\leq l$, so that $\omega\in 
V_\mathbb{Q}$. Moreover, $V_\mathbb{Z}$ is generated by the lowest weight 
vectors $\iota(e_\alpha)$ for $\alpha\in L$.
 
 If $V$ is an affine Lie algebra vertex operator algebra  
$V_{\widehat{\mathfrak{g}}}(\ell,0)_\mathbb{Z}$ or 
$L_{\widehat{\mathfrak{g}}}(\ell,0)_\mathbb{Z}$ and $\lbrace u_i\rbrace $ is a 
Chevalley basis for $\mathfrak{g}_\mathbb{Z}$ with dual basis $\lbrace 
u_i'\rbrace$ with respect to the form $\langle\cdot,\cdot\rangle$ on 
$\mathfrak{g}$, then
  \begin{equation}
 \omega=\frac{1}{2(\ell+h)}\sum_{i=1}^{\mathrm{dim}\,\mathfrak{g}}
u_i(-1)u_i'(-1)\mathbf{1}
\end{equation}
where $h$ is the dual Coxeter number of $\mathfrak{g}$. Since 
$\ell,h\in\mathbb{Z}$ and $\langle\cdot,\cdot\rangle$ is integral on 
$\mathfrak{g}_\mathbb{Z}$ so that 
$u_i'\in\mathbb{Q}\otimes_\mathbb{Z}\mathfrak{g}_\mathbb{Z}$, we have $\omega\in 
V_\mathbb{Q}$. Moreover, $V_\mathbb{Z}$ is generated by the vectors 
$\frac{x_{\alpha}(-1)^k}{k!}\mathbf{1}$ for $\alpha$ a root, $x_\alpha$ the 
corresponding root vector in the Chevalley basis of $\mathfrak{g}$, and $k\geq 
0$. The commutation relations
\begin{equation}
 [L(m), x_\alpha(-1)]=x_{\alpha}(m-1)
\end{equation}
for any $m\in\mathbb{Z}$ (see for example \cite{LL} Section 6.2), and the fact 
that $x_{\alpha}(m-1)$ commutes with $x_\alpha(-1)$, imply that for $m>0$ and 
$k\geq 0$,
\begin{equation}
L(m)\dfrac{x_\alpha(-1)^k}{k!}\mathbf{1}=\dfrac{x_{\alpha}(-1)^{k-1}}{(k-1)!}
x_\alpha(m-1)\mathbf{1}=0.
\end{equation}
Since $\frac{x_\alpha(-1)^k}{k!}\mathbf{1}$ is homogeneous of conformal weight 
$k$, this means $V_\mathbb{Z}$ is generated by lowest weight vectors for the 
Virasoro algebra.
\end{proof}

\begin{rema}
 Theorem \ref{omegatheorem} and Proposition \ref{omegaprop} generalize the 
observation made in \cite{B} that if $L$ is an
even lattice, $V_{L,\mathbb{Z}}$ can be extended to an integral form of $V_L$
containing $\omega$ if the rank of $L$ is even, and containing $2\omega$ if the
rank of $L$ is odd.
\end{rema}

Using the commutation relations (\ref{kvir}), we can use an argument similar to 
but simpler than the proof of Theorem \ref{omegatheorem} to prove:
\begin{propo}
 If $V$ is a vertex operator algebra generated by the conformal vector $\omega$, 
and $\omega$ is contained in a rational form of $V$, 
then $V$ has an integral form generated by $k\omega$ if $k\in\mathbb{Z}$ and 
$k^2 c\in 2\mathbb{Z}$. In particular, $\omega$ generates an integral form of 
$V$ 
if and only if $c\in 2\mathbb{Z}$.
\end{propo}

If $L$ is an even lattice, then $V_{L,\mathbb{Z}}$ may already contain $\omega$.
Recall that if $\left\lbrace \alpha_1,\ldots 
,\alpha_l\right\rbrace $ is a base for $L$ and $\left\lbrace 
\alpha_1',\ldots ,\alpha_l'\right\rbrace $ is the corresponding dual basis with
respect to 
$\left\langle \cdot ,\cdot\right\rangle $, then
\begin{equation}
 \omega =\frac{1}{2}\sum_{i=1}^l \alpha_i(-1)\alpha_i'(-1)\mathbf{1}\in V^0_2,
\end{equation} 
and the central charge of $V_L$ is $l$, the rank of $L$. The ``if'' direction of the following proposition was observed in \cite{BR1}:
\begin{propo}
 If $L$ is an even lattice, the integral form $V_{L,\mathbb{Z}}$ of $V_L$ 
contains $\omega$ if and only if $L$ is self-dual.
\end{propo}
\begin{proof}
Suppose  $\left\lbrace \alpha_1,\ldots ,\alpha_l\right\rbrace $ is a base for 
$L$. We know from Proposition \ref{lattbasisprop} that an integral basis for 
$V_{L,\mathbb{Z}}\cap V^0_2$ consists of distinct coefficients of monomials of 
degree $2$ in products of the form
\begin{equation}\label{vzerotwo}
 E^-(-\alpha_i,x_1)E^-(-\alpha_j,x_2)\mathbf{1} 
\end{equation}
where $1\leq i,j\leq l$. Since
\begin{eqnarray}
 E^-(-\alpha,x) & = & \mathrm{exp}\left( \sum_{n>0} \frac{\alpha(-n)}{n}
x^{n}\right) \nonumber\\
& = & 1+\alpha(-1)x+\left( \frac{\alpha(-2)+\alpha(-1)^2}{2} \right)x^2+\ldots
,\\\nonumber
\end{eqnarray}
the distinct coefficients of monomials of degree $2$ in (\ref{vzerotwo}) are
\begin{equation}\label{vzerotwobasis}
 \alpha_i(-1)\alpha_j(-1)\mathbf{1},\,\,\,\,\frac{\alpha_i(-2)+\alpha_i(-1)^2}
{2}\mathbf{1}
\end{equation}
where $1\leq i\leq j\leq l$. We can take these quadratic polynomials as a base 
for 
$V_{L,\mathbb{Z}}\cap V^0_2$.

Now suppose $\left\lbrace \alpha_1',\ldots ,\alpha_l'\right\rbrace $ is a 
basis of $\mathfrak{h}$ dual to $\left\lbrace \alpha_1,\ldots 
,\alpha_l\right\rbrace $, and write $\alpha_i'=\sum_{j=1}^l c_{ji}\alpha_j$ 
where $c_{ji}\in\mathbb{C}$. Then
\begin{eqnarray}
 \omega & = & \frac{1}{2}\sum_{i=1}^l
\alpha_i(-1)\alpha_i'(-1)\mathbf{1}=\frac{1}{2}\sum_{i,j=1}^l
c_{ji}\alpha_i(-1)\alpha_j(-1)\mathbf{1}\nonumber\\
& = & \sum_{i=1}^l \frac{c_{ii}}{2}\alpha_i(-1)^2\mathbf{1}+\sum_{i<j}
\frac{c_{ij}+c_{ji}}{2}\alpha_i(-1)\alpha_j(-1)\mathbf{1}.\\\nonumber
\end{eqnarray}
In view of the base (\ref{vzerotwobasis}) for $V_{L,\mathbb{Z}}\cap V^0_2$, we
see that 
$\omega\in V_{L,\mathbb{Z}}$ if and only if $c_{ii},c_{ij}+c_{ji}\in 
2\mathbb{Z}$ for all $i$ and $j\neq i$.

Since $\left\lbrace \alpha_1',\ldots ,\alpha_l'\right\rbrace $ is a dual 
basis,
\begin{equation}
 \left\langle \alpha_i',\alpha_j'\right\rangle =\left\langle
\alpha_i',\sum_{k=1}^l c_{kj}\alpha_k\right\rangle =c_{ij}.
\end{equation}
Since $\langle\cdot,\cdot\rangle$ is symmetric, we have $c_{ij}=c_{ji}$. 
Consequently, 
$\omega\in V_{\mathbb{Z}}$ if and only if $c_{ii}\in 2\mathbb{Z}$ for all $i$ 
and $c_{ij}\in\mathbb{Z}$ for $i\neq j$. If $L$ is self-dual, each 
$\alpha_i'\in L$, so each $c_{ij}\in\mathbb{Z}$; also, since $L$ is even, 
$c_{ii}=\left\langle \alpha_i',\alpha_i'\right\rangle\in 2\mathbb{Z}$ for 
each $i$. Conversely, if each $c_{ij}\in\mathbb{Z}$, then each $\alpha_i'\in 
L$, so $L$ is self-dual. Thus we see that $\omega\in V_{L,\mathbb{Z}}$ if and 
only if $L$ is self-dual.
\end{proof}

\begin{exam}
 If $L$ is the root lattice of $E_8$ or the Leech lattice, then $\omega\in 
V_{L,\mathbb{Z}}$.
\end{exam}

\section{Integral forms in contragredient modules}

The following results generalize Lemma 6.1, Lemma 6.2, and
Remark 6.3 in \cite{DG} to the context of modules and contragredient modules 
for 
a vertex operator algebra $V$. Also, we apply these results to the affine Lie 
algebra vertex operator algebra integral forms 
$V_{\widehat{\mathfrak{g}}}(\ell,0)_\mathbb{Z}$ and 
$L_{\widehat{\mathfrak{g}}}(\ell,0)$, cases not considered 
in \cite{DG}. Suppose $V$ has an integral form 
$V_{\mathbb{Z}}$ and $W$ is a $V$-module with integral form $W_{\mathbb{Z}}$. 
Then there is an integral form $W'_{\mathbb{Z}}$ in the contragredient module 
$W'$ given by
\begin{equation}
 W'_{\mathbb{Z}}=\left\lbrace w'\in W'\mid\left\langle
w',w\right\rangle\in\mathbb{Z}\;\mathrm{for}\,w\in W_{\mathbb{Z}} \right\rbrace
.
\end{equation}
We would like $W'_\mathbb{Z}$ to be a module for $V_\mathbb{Z}$.
\begin{propo}
 Suppose $V_\mathbb{Z}$ is invariant under $\frac{L(1)^n}{n!}$ for $n\geq 0$.
Then $W'_\mathbb{Z}$ is invariant under the action of $V_\mathbb{Z}$.
\end{propo}
\begin{proof}
 By the definition of the module action on $W'$, if $v\in V_\mathbb{Z}$, $w\in
W_\mathbb{Z}$ and $w'\in W'_\mathbb{Z}$,
\begin{equation}
 \left\langle Y(v,x)w',w\right\rangle = \left\langle w',Y(e^{xL(1)}
(-x^{-2})^{L(0)} v,x^{-1})w\right\rangle 
\end{equation}
Since $\frac{L(1)^n}{n!}$ leaves $V_\mathbb{Z}$ invariant, $e^{xL(1)} 
(-x^{-2})^{L(0)} v\in V_\mathbb{Z}\left[  x,x^{-1} \right].$ Since also 
$V_\mathbb{Z}$ leaves $W_\mathbb{Z}$ invariant, by definition $ \left\langle 
Y(v,x)w',w\right\rangle\in\mathbb{Z}[[x,x^{-1}]]$ for any $w$ and so
$Y(v,x)w'\in 
W'_\mathbb{Z}\left[ \left[ x,x^{-1}\right] \right] $ as desired.
\end{proof}
\begin{propo}
 If $V_\mathbb{Z}$ is generated by vectors $v$ such that $L(1)v=0$, then 
$V_\mathbb{Z}$ is invariant under $\frac{L(1)^n}{n!}$ for $n\geq 0$.
\end{propo}
\begin{proof}
 We will use the $L(1)$-conjugation formula proved in \cite{FHL}:
\begin{equation}
 e^{yL(1)}Y(v,x)e^{-yL(1)}=Y\left(
e^{y(1-xy)L(1)}(1-xy)^{-2L(0)}v,\frac{x}{1-xy}\right). 
\end{equation}
If $L(1)v=0$, this equation simplifies to 
\begin{equation}
  e^{yL(1)}Y(v,x)e^{-yL(1)}=Y\left( (1-xy)^{-2L(0)}v,\frac{x}{1-xy}\right) .
\end{equation}
Thus if a vector $v\in V_\mathbb{Z}$ is a coefficient of a monomial in
\begin{equation}\label{gen}
 Y(v_1,x_1)\cdots Y(v_k,x_k)\mathbf{1}
\end{equation}
where $L(1)v_i=0$, then $\frac{L(1)^n}{n!}v$ is a coefficient of a monomial in
\begin{eqnarray}
 && e^{yL(1)}Y(v_1,x_1)\ldots Y(v_k,x_k)\mathbf{1}= \nonumber\\
&& Y\left( (1-x_1y)^{-2L(0)}v_1,\frac{x_1}{1-x_1y}\right)\cdots Y\left(
(1-x_ky)^{-2L(0)}v_k,\frac{x_k}{1-x_ky}\right)\mathbf{1} \\\nonumber
\end{eqnarray}
Since the expansion of $(1-xy)^m$ for any integer $m$ has integer 
coefficients, all coefficients of monomials on the right side lie in 
$V_\mathbb{Z}$. Hence $\frac{L(1)^n}{n!} v\in V_\mathbb{Z}$. Thus if 
$V_\mathbb{Z}$ is spanned by coefficients of monomials of the form in 
(\ref{gen}), then $\frac{L(1)^n}{n!}$ leaves $V_\mathbb{Z}$ invariant for any 
$n\geq 0$.
\end{proof}

\begin{rema}
 If an integral form $V_\mathbb{Z}$ of $V$ is generated by lowest weight vectors 
for the Virasoro algebra, then contragredients of $V_\mathbb{Z}$-modules are 
$V_\mathbb{Z}$-modules. In particular, this holds for lattice and affine Lie 
algebra vertex operator algebras by Proposition \ref{omegaprop}.
\end{rema}

 Suppose $V$ is equivalent as $V$-module to its contragredient $V'$. This is 
the case if and only if there is a non-degenerate bilinear form $\left(\cdot 
,\cdot \right)$  on $V$ that is invariant in the sense that
\begin{equation}
 \left( Y(u,x)v,w\right) =\left( v, Y(e^{xL(1)} (-x^{-2})^{L(0)}
u,x^{-1})w\right)  
\end{equation}
for $u,v,w\in V$ (see \cite{FHL}). 
Invariant forms on $V$ are in one-to-one correspondence with linear 
functionals on $V_0/L(1)V_1$ (\cite{Li}). In the case of lattice and affine 
Lie algebra vertex operator algebras, invariant forms are unique up to scale 
since $V_0=\mathbb{C}\mathbf{1}$ and $L(1)V_1=0$. (In the afffine Lie algebra 
case, $V$ has a non-degenerate invariant bilinear form only when 
$V=L_{\widehat{\mathfrak{g}}}(\ell,0)$.) Given a choice of invariant 
bilinear form and an integral form $V_\mathbb{Z}$ of a vertex operator algebra 
$V$, the contragredient module $V'_{\mathbb{Z}}$ 
may be identified with another lattice spanning $V$ that is invariant under 
the action of $V_\mathbb{Z}$. However, $V'_\mathbb{Z}$ need not be an 
integral form of $V$ as a vertex algebra, because it may not be closed under
vertex algebra 
products.
\begin{propo}
 Suppose $V$ is equivalent to $V'$ as $V$-module and $V$ has an integral form
$V_\mathbb{Z}$ preserved by $\frac{L(1)^n}{n!}$ for $n\geq 0$; also assume
$V_0=\mathbb{C}\mathbf{1}.$ Identify $V'_\mathbb{Z}$ with a lattice in $V$ using
a non-degenerate invariant form $(\cdot,\cdot)$ such that
$(\mathbf{1},\mathbf{1})\in\mathbb{Z}\setminus\left\lbrace 0 \right\rbrace $.
Then $V_\mathbb{Z}\subseteq V'_\mathbb{Z}.$
\end{propo}
\begin{proof}
 The integral form of $V'$ is identified as: 
\begin{equation}
 V'_\mathbb{Z} =\left\lbrace v'\in V\mid (v',v)\in\mathbb{Z}\;\mathrm{for}\,
v\in V_\mathbb{Z}\right\rbrace .
\end{equation}
Thus we need to show that if $u,v\in V_\mathbb{Z}$, then $(u,v)\in\mathbb{Z}$.
We have:
\begin{eqnarray}
 (u,v) & = & \mathrm{Res}_x\, x^{-1} (Y(u,x)\mathbf{1},v)\nonumber\\
       & = & \mathrm{Res}_x\, x^{-1} (\mathbf{1}, Y(e^{xL(1)} (-x^{-2})^{L(0)}
u,x^{-1})v)\nonumber\\
       & = & (\mathbf{1}, c\mathbf{1})=c\, (\mathbf{1},\mathbf{1})\\\nonumber
\end{eqnarray}
where $c\mathbf{1}=\mathrm{Res}_x\, x^{-1} Y(e^{xL(1)} (-x^{-2})^{L(0)} 
u,x^{-1})v$ (since the residue is indeed in $V_0$). But all coefficients of
$Y(e^{xL(1)} (-x^{-2})^{L(0)} 
u,x^{-1})v$ are in $V_\mathbb{Z}$ because $V_\mathbb{Z}$ is closed under 
vertex operators and invariant under $\frac{L(1)^n}{n!}$ for $n\geq 0$. Thus 
$c\mathbf{1}\in 
V_0\cap V_\mathbb{Z}=\mathbb{Z}\mathbf{1}$, by Proposition \ref{spanofvacuum}
and so $c\, (\mathbf{1},\mathbf{1})$
is an integer.
\end{proof}
\begin{rema}
 The hypotheses of this proposition are satisfied if $V$ is either a lattice or 
irreducible
affine Lie algebra vertex operator algebra and we choose, for instance,
$(\mathbf{1},\mathbf{1})=\pm 1$.
\end{rema}

\noindent {\small \sc Department of Mathematics, Rutgers University,
110 Frelinghuysen Rd., Piscataway, NJ 08854-8019}
\vspace{1em}

\noindent {\em E-mail address}: rhmcrae@math.rutgers.edu

\end{document}